\documentclass[12pt]{amsart}
\usepackage[active]{srcltx}
\usepackage{latexsym}
\usepackage{amssymb}
\usepackage{verbatim}
\usepackage{tikz}
\usepackage{enumerate}

\usepackage{amsmath}

\newcommand{\ignore}[1]{}

\newcommand{\hide}[1]{}

\DeclareMathOperator{\Der}{Der}

\newcommand{\F}{\mathbb F}

\newcommand{\N}{\mathbb N}

\newcommand{\Z}[0]{\mathbb Z}

\newtheorem{dummy}{Dummy}

\newtheorem{lemma}[dummy]{Lemma}

\newtheorem{thm}[dummy]{Theorem}

\newtheorem{prop}[dummy]{Proposition}

\theoremstyle{definition}

\theoremstyle{remark}
\newtheorem{rem}[dummy]{Remark}
\newtheorem*{rem*}{Remark to ourselves}

\begin{document}

\bibliographystyle{alpha}

\title{Graded Lie algebras of maximal class of type $p$}

\author{Valentina Iusa}
\address{Charlotte Scott Research Centre for Algebra\\
University of Lincoln \\
Brayford Pool
Lincoln, LN6 7TS\\
United Kingdom}
\email{valentina\_iusa@hotmail.it}

\author{Sandro Mattarei}
\address{Charlotte Scott Research Centre for Algebra\\
University of Lincoln \\
Brayford Pool
Lincoln, LN6 7TS\\
United Kingdom}
\email{smattarei@lincoln.ac.uk}

\author{Claudio Scarbolo}
\address{Dipartimento di Matematica\\
Universit\`a degli Studi di Trento\\
via Sommarive 14\\
I-38123 Povo (Trento)\\
Italy}
\email{claudio.scarbolo-1@unitn.it}



\begin{abstract}
The algebras of the title are infinite-dimensional graded Lie algebras
$
L= \bigoplus_{i=1}^{\infty}L_i,
$
over a field of positive characteristic $p$,
that are generated by an element of degree $1$ and an element of degree $p$, and satisfy
$[L_i,L_1]=L_{i+1}$ for $i\ge p$.
In case $p=2$ such algebras were classified by Caranti and Vaughan-Lee in 2003.
We announce an extension of that classification to arbitrary prime characteristic, and prove several major steps in its proof.
\end{abstract}

Author Accepted Manuscript

Published in Journal of Algebra \textbf{588} (2021), 77-117.

\verb#https://dx.doi.org/10.1016/j.jalgebra.2021.08.013#

\bigskip

\maketitle

\section{Introduction}\label{sec:intro}
A Lie algebra $L$ is nilpotent if its Lie power $L^{c+1}$ vanishes for some positive integer $c$,
and the smallest such integer is the nilpotency class of $L$.
If $L$ has finite dimension $d\ge 2$, then it is well known that $c\le d-1$.
The difference $d-c$ is the {\em coclass} of $L$.
Lie algebras of coclass one are said to be {\em of maximal class,} but
they are often called {\em filiform} in the literature.

The concept of coclass readily extends to infinite-dimensional Lie algebras
by taking the limit superior of the coclass of their quotients $L/L^{i+1}$.
In particular, a residually nilpotent infinite-dimensional Lie algebra $L$ has coclass one or, by a slight abuse of terminology, has maximal class,
if $\dim(L/L^2)=2$ and $\dim(L^i/L^{i+1})=1$ for all $i>1$.

Interest in coclass originated from group theory, where celebrated {\em coclass conjectures} of Leedham-Green and Newman~\cite{LGN} initiated a paradigm shift
in the study of finite $p$-groups, which are notoriously hard to classify in terms of nilpotency class,
directing attention towards pro-$p$ groups, which encode entire families of finite $p$-groups as their quotients.
The most easily stated of their five conjectures is conjecture C: pro-$p$ groups of finite coclass are soluble.
After all coclass conjectures were established, with the final steps in~\cite{L-G} and~\cite{Sha:coclass},
Shalev and Zelmanov used that model to develop a coclass theory for Lie algebras of characteristic zero in~\cite{ShZe:narrow}.

Shalev and Zelmanov focused on $\N$-graded Lie algebras, meaning with a grading of the form
$L= \bigoplus_{i=1}^{\infty}L_i$
(or, equivalently, $\mathbb{Z}$-graded but with $L_i=0$ for $i \le 0$).
Furthermore, a natural assumption coming from the viewpoint of pro-$p$ groups was that
$L$ should be generated by their first homogeneous component $L_1$.
It will be apparent later in this Introduction that this assumption is also necessary for their coclass theory to hold.
Working over a field of zero characteristic was also necessary,
as shown by insoluble graded Lie algebras of maximal class constructed by Shalev in~\cite{Sha:max} from certain finite-dimensional simple modular Lie algebras.

A study of the modular case was initiated by Caranti, Mattarei and Newman in~\cite{CMN},
illustrating how intricate the picture is in positive characteristic.
In particular, it showed that over any field $F$ of positive characteristic there are $|F|^{\aleph_0}$ infinite-dimensional
non-isomorphic graded Lie algebras of maximal class, generated by $L_1$,
arising through procedures called {\em inflation} and {\em deflation} from the algebras described by Shalev in~\cite{Sha:max},
possibly applied countably many times, and one further limit construction.
Remarkably, Caranti and Newman were able to prove in~\cite{CN} that
suitable application of the machinery developed in~\cite{CMN} is able to account for all graded Lie algebra of maximal class, generated by $L_1$,
in odd characteristic.
A matching result in characteristic $2$, which presents further complications, was obtained in~\cite{Ju:maximal}.
Thus, modular graded Lie algebras of maximal class generated by $L_1$ can be considered as classified, although the classification statement
may look unconventional as it involves the possible application of certain steps countably many times.
This classification is also important because classifications of other classes of Lie algebras rely on it as an essential ingredient.
In particular, this is the case for the algebras
studied in this paper, and also for {\em thin} Lie algebras,
see~\cite{AviMat:A-Z} and the references therein.

Other $\N$-gradings of Lie algebras of maximal class are of great interest, where
the Lie algebra is not generated by its first homogeneous component $L_1$.
One notable example in characteristic zero is the positive part of the Witt algebra $\Der(F[X])$,
hence the subalgebra $W$ with graded basis given by
$E_i=x^{i+1}\partial/\partial x$
for $i>0$ (having degree $i$), with Lie product
$[E_i,E_j]=(j-i)E_{i+j}$.
It is easily seen that $W$, which we will discuss further at the end of Section~\ref{sec:sequence}, is a Lie algebra of maximal class,
and is generated by $E_1$ and $E_2$, which have different degrees in the natural grading exhibited here.
Shalev and Zelmanov's showed in~\cite[Theorem~7.1]{ShZe:narrow-Witt} that besides $W$ there are precisely two further
(infinite-dimensional) graded Lie algebras of maximal class $L$, over any given field of characteristic zero,
admitting an $\N$-grading such that $L$ is generated by $L_1$ and $L_2$.
Those two are both soluble, with a maximal abelian ideal of codimension $1$ or $2$,
and will be also recalled in Section~\ref{sec:sequence}.
Note that those conditions on $L$ in~\cite[Theorem~7.1]{ShZe:narrow-Witt} allow the possibility of $L_1$ being two-dimensional and generating $L$ by itself,
and that classification result is up to isomorphism of Lie algebras, not of graded Lie algebras:
the Lie algebra of maximal class with an abelian maximal ideal is counted once in the result, but admits two such gradings.

To make progress in positive characteristic, where the landscape is distinctively more complex,
it is best to adopt a more restrictive definition (see Remark~\ref{rem:other_types} for a comparison).
Graded Lie algebras of maximal class $L$ which are generated by $L_1$, as in~\cite{CMN,CN,Ju:maximal}, will be called {\em algebras of type $1$} here.
Then, necessarily, $L_1$ will be two-dimensional, and $L_i$ will be one-dimensional for all $i>1$ (under our blanket infinite-dimensionality assumption).
We define an {\em algebra of type $n$,} for $n>1$, as a graded Lie algebra $L=\bigoplus_{i=1}^{\infty}L_i$
with $\dim L_i=0$ for $1<i<n$ and $\dim L_i=1$ otherwise,
such that $[L_i,L_1]=L_{i+1}$ for all $i\ge n$.
Such $L$ is a Lie algebra of maximal class, and is generated by $L_1$ and $L_n$, both one-dimensional.
However, beware that the latter assumptions alone on $L$ (of infinite dimension) are not sufficient to ensure
that $L$ is an algebra of type $n$, see Remark~\ref{rem:other_types}.
Thus, \cite[Theorem~7.1]{ShZe:narrow-Witt} showed, in particular, that over a field of characteristic zero there are precisely three algebras of type $2$,
up to isomorphism of graded Lie algebras.

A classification of algebras of type $n$ is bound to involve all the complexity of algebras of type $1$ due to the following simple observation.
An algebra of type $1$ is called {\em uncovered} if there is an element $z\in L_1$ such that $[L_i,z]=L_{i+1}$ for all $i>1$.
This condition can always be achieved by extending the ground field, and is automatic if that is uncountable.
If $L$ is an uncovered algebra of type $1$, then the subalgebra generated by $z$ and $L_n$ is an algebra of type $n$.
Note that different choices of $z$ in $L_1$ may lead to non-isomorphic subalgebras of type $n$.

Caranti and Vaughan-Lee classified algebras of type $2$ over a field of odd characteristic in~\cite{CVL00}.
They showed that all algebras of type $2$ which are not graded subalgebras of algebras of type $1$ in the way described above are soluble,
and belong to an explicitly described family (with an additional family appearing in characteristic three).
They anticipated that the case of characteristic two would be more complicated,
but then a classification in that case, which they achieved in~\cite{CVL03}, actually turned out to be simpler to state.
It is the special case $p=2$ of Theorem~\ref{thm:classification} below.

What makes the classification of algebras of type $n$ more approachable when $n$ equals the characteristic is
a generalization of a {\em translation} process discovered in~\cite{CVL03}, which we now briefly describe.
To start with, having chosen a nonzero element $z$ in $L_1$, of an algebra of type $n>1$,
we extend that to a graded basis of $L$ by choosing a nonzero element of $L_n$, and then recursively setting $e_i=[e_{i-1},z]$ for $i>n$.
The algebra $L$ is then completely described by the sequence of scalars $(\beta_i)_{i>n}$, where $[e_i,e_n]=\beta_ie_{i+n}$,
because knowing the adjoint action of the generators $z$ and $e_n$ determines
the full multiplication table of $L$ via recursive use of the Jacobi identity.
When $n$ equals the characteristic $p$, adding a constant value $\delta$ to each $\beta_i$ in the sequence
produces the sequence of another algebra of type $p$, called a {\em translate} of $L$ and denoted by $L(\delta)$
(see Proposition~\ref{prop:translation}).
We are now ready to announce the following classification result, whose special case $p=2$ was proved in~\cite{CVL03}.

\begin{thm}\label{thm:classification}
Let $L$ be an algebra of type $p$, over a field of characteristic $p>0$.
If $p>2$ assume $(L^2)^2 \subseteq L^{3p+3}$.
Then, $L$ is either a graded subalgebra of an uncovered algebra of type $1$, or a translate of that,
or $L$ belongs to an explicitly described countable family $\mathcal{E}$.
\end{thm}

The family $\mathcal{E}$ is empty when $p=2$, as in~\cite{CVL03}.
Theorem~\ref{thm:classification} was proved by the third author in his PhD thesis~\cite{Sca:thesis}, under supervision of the second author.
The main goal of this paper is to prove some major steps in the proof of Theorem~\ref{thm:classification}.

Before describing the results of this paper in some detail we briefly digress on algebras of arbitrary type $n>1$.
For $n<p$ those were investigated by Ugolini in his PhD thesis~\cite{Ugo:thesis}, written under the joint supervision of Andrea Caranti and the second author.
Ugolini's work predated, and influenced, Scarbolo's work on algebras of type $p$,
but it turned out that working in such generality is fraught with technical difficulties, especially when $n$ gets close to $p$.
Nevertheless, Ugolini obtained important structural information on algebras of type $n$,
and even a partial classification if $p>4n$.
A new take on part of Ugolini's results is presented in~\cite{MatUgo:type_n}.
Our technical results in Sections~\ref{sec:constituents}--\ref{sec:characterization} cover several basic aspects of algebras of type $n$,
before we set $n=p$ and proceed towards the main results of this paper.

The original version~\cite[Theorem 2.6]{Sca:thesis} of Theorem~\ref{thm:classification}
did not assume the hypothesis $(L^2)^2\subseteq L^{3p+3}$ (which was unnecessary in~\cite{CVL03}, where $p=2$).
However, the proof given in~\cite{Sca:thesis} appears incomplete without that hypothesis.
That was discovered during substantial simplification work of a portion of Scarbolo's proof
achieved by the first author in~\cite{Iusa:thesis}, again under the supervision of the second author.
We will comment further below on the significance of this additional hypothesis.

As we mentioned earlier, to an algebra of type $n>1$ we associate a sequence of scalars $(\beta_i)_{i>n}$,
which we call {\em the sequence} of $L$.
Depending on a choice of generators $z$ and $e_n$ the sequence is only unique to $L$ up to scalar multiples.
For algebras of type $1$
that was called {\em the sequence of two-steps centralizers,}
reminiscent of its group theoretic origins in N.~Blackburn's theory of $p$-groups of maximal class.
As was the case for algebras of type $1$ in~\cite{CMN,CN,Ju:maximal},
the sequence $(\beta_i)_{i>n}$
can be split into a union of adjacent blocks called the {\em constituents} of $L$.
Assuming $L$ not metabelian, which excludes precisely two algebras of each type $n>1$,
such blocks are all finite, and their lengths $\ell_1,\ell_2,\ldots$
(with an {\em ad hoc} definition in case of $\ell_1$)
play an important role in the theory.
In particular, the length $\ell=\ell_1$ of the first constituent, which has an {\em ad hoc} definition,
is the most fundamental invariant which one associates to $L$.
In case of algebras of type $1$ the first constituent length (if finite)
is always twice a power $q$ of the characteristic~\cite[Theorem~5.5]{CMN}.

The constituent lengths of an algebra of type $n$ (including $1$) admit the following simple characterization
(Proposition~\ref{prop:constituent_lengths}) under the assumption that $\beta_{n+1}$ vanishes:
$\ell=\ell_1=\dim\bigl(L^2/(L^2)^2\bigr)+n$,
and
$\ell_r=\dim\bigl((L^2)^r/(L^2)^{r+1}\bigr)$ for $r>1$.
The assumption $\beta_{n+1}=0$ is inessential when $n=1$, as is naturally achieved by an appropriate choice of the generators.
Fortunately, when the type $n$ equals the characteristic it can be achieved by passing to a translate of $L$.
Thus, in the context of Theorem~\ref{thm:classification}, after taking such a translate of $L$ to satisfy $\beta_{p+1}=0$,
the hypothesis $(L^2)^2 \subseteq L^{3p+3}$ on $L$, which is unaffected by passing to the translate,
reads $\ell>4p$ in terms of that translate.
Under that hypothesis, in the following result we determine all the possibilities for $\ell$.

\begin{thm}\label{thm:first_length}
Let an algebra of type $p$, over a field of odd characteristic $p$, have its first constituent of finite length $\ell>4p$.
Then
\begin{itemize}
\item either $\ell=2q$, where $q$ is a power of $p$,
\item or $q<\ell\le q+p$, where $q$ is a power of $p$ and $\ell$ is even.
\end{itemize}
\end{thm}

Each of the possibilities for $\ell$ listed in Theorem~\ref{thm:first_length} does occur for some of the algebras of Theorem~\ref{thm:classification}.
In particular, the first constituent of $L$ has length $\ell=2q$ or $\ell=q+p$ whenever $L$
is a graded subalgebra of an uncovered algebra of type $1$, or a translate of that (see Section~\ref{sec:type_p}).
The remaining possibilities occur for the algebras of the exceptional family $\mathcal{E}$.
The analogue~\cite[Proposition~3.1]{CVL03} of Theorem~\ref{thm:first_length} in characteristic two does not require any assumption on $\ell$.

Theorem~\ref{thm:first_length} is the most substantial result of this paper.
Our proof is a considerably streamlined version of its original proof in~\cite{Sca:thesis}, and offers a more transparent route to the desired conclusion.
Our main improvement is the use of one particular relation holding in $L$, rather than several as in~\cite{Sca:thesis},
and translating the information it provides to a polynomial setting
(which is Equation~\eqref{eq:range} in Theorem~\ref{thm:polynomials}, where $k=\ell-p+1$).
This follows an approach pioneered in~\cite{Mat:chain_lengths}, which we briefly recall at the end of Section~\ref{sec:type_p}.
Although the polynomial formulation of this information
still requires some nontrivial arguments in order to draw the desired conclusions on $\ell$,
it provides in itself a plausible indication as to why precisely those values for $\ell$ occur in the conclusion of Theorem~\ref{thm:first_length}.

A similar result to Theorem~\ref{thm:first_length} is proved in~\cite{MatUgo:type_n} for algebras of type $n<p$,
in which case the proof presents further challenges.

Our next result forms another portion of the proof of Theorem~\ref{thm:classification}.
It characterizes those algebras of type $p$ which occur as graded subalgebras of uncovered algebras of type $1$,
in terms of their first constituent length alone.

\begin{thm}\label{thm:length2q}
Let $L$ be an algebra of type $p$, over a field of odd characteristic $p$, with first constituent length $\ell=2q$, where $q>p$ is a power of $p$.
Then $L$ is a graded subalgebra of an uncovered algebra of type $1$.
\end{thm}

The exclusion of $q=p$ in Theorem~\ref{thm:length2q} is because algebras of type $p$ with $\ell=2p$
may also occur as (proper) translates of algebras of type $1$ with a longer first constituent (see Proposition~\ref{prop:translated_subalgebra}).

Our final result is an explicit construction of the exceptional algebras of the family $\mathcal{E}$ of Theorem~\ref{thm:classification}.
We summarize here the conclusion as an existence theorem, and refer to Section~\ref{sec:exceptional} for further structural information on them.

\begin{thm}\label{thm:exceptional}
Let $q$ be a power of an odd prime $p$, with $q>p$.
For every $0<m<p-1$ there is a soluble algebra $L$ of type $p$, over a field of odd characteristic $p$,
satisfying the following:
\begin{itemize}
\item
the length $\ell$ of the first constituent of $L$ equals $q+m$ if $m$ is odd, and $q+m+1$ if $m$ is even;
\item
all constituents past the first have length $q$,
except for the second constituent if $m$ is even, which has then length $q-1$.
\end{itemize}
\end{thm}

Note that the algebras of Theorem~\ref{thm:exceptional} attain all even values of $\ell$ in the interval
$q<\ell\le q+p$ of~Theorem~\ref{thm:first_length} with the only exception of $q+p$.
This last value of $\ell$ occurs for certain translates of subalgebras of type $p$ of an algebra of type $1$,
as we show in Proposition~\ref{prop:translated_subalgebra}.

Theorem~\ref{thm:first_length} shows that a proof of Theorem~\ref{thm:classification} can be divided into three parts according to the value of $\ell$.
Under hypotheses and notation as in Theorem~\ref{thm:classification} those parts are proofs of the following assertions:
\begin{itemize}
\item
if $\ell=2q$ then $L$ is a graded subalgebra of an algebra of type $1$;
\item
if $\ell=q+p$ then a suitable translate of $L$ is a graded subalgebra of an algebra of type $1$;
\item
if $q<\ell<q+p$ then $L$ is isomorphic to one of the exceptional algebras in the family $\mathcal{E}$ described in Theorem~\ref{thm:exceptional}.
\end{itemize}
Note that the hypothesis $\ell>4p$ of Theorem~\ref{thm:classification} ensures $q>p$.
The first of those three assertion is a consequence of Theorem~\ref{thm:length2q}.
Proofs of the remaining two assertions will be presented in a separate paper.

We now outline the structure of the paper.
In Section~\ref{sec:sequence} we introduce gradings of Lie algebras of maximal class,
defining in particular algebras of type $1$, and of type $n>1$.
Note that a Lie algebra of maximal class which is graded over the positive integers
and generated by an element of degree $1$ and one of degree $n$,
need not be an algebra of type $n$, as we point out in Remark~\ref{rem:other_types}.
We also define {\em the sequence} of an algebra $L$ of type $n>1$,
which is the analogue of the {\em sequence of two-step centralizers} for algebras of type $1$.

The sequence of $L$ can be partitioned into (finite) adjacent blocks called {\em constituents,}
which we introduce in Section~\ref{sec:constituents} extending the corresponding definition for algebras of type $1$.
The definition is such that the constituents of an algebra $L$ of type $n$ which is a graded subalgebra of an algebra $N$ of type $1$
naturally correspond to the constituents of $N$, with matching lengths,
provided that the first constituent of $N$ has length at least $2n$.
Only the last $n$ entries of any constituent of an algebra $L$ of type $n$ can be nonzero,
and if $L$ occurs as a graded subalgebra of an algebra of type $1$ then those last $n$ entries take a particular form given by Equation~\eqref{eq:ordinary_from_type_1}.
We prove the converse implication in Proposition~\ref{lemma:Ugolini}, extending a result in~\cite{Ugo:thesis}.
We are grateful to Simone Ugolini for permission to include that result, albeit in a strengthened form.

Constituents, and their lengths, were fundamental in the theory of algebras of type $1$ in~\cite{CMN,CMNS,Ju:maximal}, and so they remain for type $n$.
In Section~\ref{sec:const_lengths} we establish some general facts on the constituent lengths $\ell=\ell_1,\ell_2,\ell_3,\ldots$ of an algebra $L$ of type $n$.
Assuming $L$ not metabelian, we show that $\ell$ is even, and that all constituent lengths satisfy $\ell/2\le\ell_r\le\ell$.
This information will later be used in our proofs of Theorems~\ref{thm:first_length} and~\ref{thm:length2q}, which concern algebras of type $p$, the characteristic,
but requires no more effort to prove for algebras of arbitrary type $n$.
In Section~\ref{sec:characterization} we show how the constituent lengths of an algebra $L$ of type $n$ can be computed from relative codimensions
in the sequence of Lie powers of $L^2$,
under a certain assumption already discussed in this Introduction before stating Theorem~\ref{thm:first_length}.

In Section~\ref{sec:type_p} we move closer to the main goals of this paper by restricting our attention to algebras of type $p$.
We introduce translates $L(\delta)$ of an algebra $L$ of type $p$ and we discuss their constituent lengths.
We also sketch a reviewed proof from~\cite{Mat:chain_lengths} of a basic fact known since~\cite{CMN},
that the first constituent length of algebras of type $1$ has length of the form $2q$.
That reviewed approach informs our strategy of proof of Theorem~\ref{thm:first_length} through a polynomial formulation.

Section~\ref{sec:polynomials} contains a key step in our proof of Theorem~\ref{thm:first_length},
but is entirely concerned with answering a polynomial question which may well be of independent interest.

Section~\ref{sec:first_length_proof} contains a proof of Theorem~\ref{thm:first_length}.
The main part of the argument reduces the problem to an application of the polynomial results of Section~\ref{sec:polynomials},
restricting the possibilities for $\ell$ to those listed in Theorem~\ref{thm:first_length}, plus a few spurious ones,
which are then ruled out by further Lie algebra calculations.

Section~\ref{sec:length2q} contains a proof of Theorem~\ref{thm:length2q}.

Finally, in Section~\ref{sec:exceptional} we construct the algebras in the exceptional family $\mathcal{E}$ of Theorem~\ref{thm:classification},
thereby proving Theorem~\ref{thm:exceptional}.
The construction and structural analysis of those exceptional algebras ran over ten pages in~\cite[Chapter~5]{Sca:thesis},
and was based on a matricial construction which extended one in~\cite[Section~8]{CVL00}.
We provide a new construction based on derivations of a ring of divided powers, which allows a much more compact treatment
in a couple of pages of Section~\ref{sec:exceptional}.

\section{Graded Lie algebras of maximal class}\label{sec:sequence}

A finite-dimensional Lie algebra $L$, with $\dim L>1$, is said to be {\em of maximal class} if it is nilpotent of nilpotency class
as large as it can be compatibly with its dimension, namely, precisely one less than that.
This means that $L/L^2$ is two-dimensional, and $L^i/L^{i+1}$ is one-dimensional for $1<i<\dim L$.
Thus, $\dim(L/L^i)=i$ for $1<i<\dim L$.
Here $L^i$ denotes the $i$th Lie power of $L$, which means the same as the $i$th term of the lower central series of $L$.
Lie algebras of maximal class are also called {\em alg\'{e}bres filiformes} in the literature.
By extension, a residually nilpotent, infinite-dimensional Lie algebra $L$ is said to be {\em of maximal class} if $\dim(L/L^i)=i$ for each $i>1$.
Carrying over standard terminology for pro-$p$ groups, they are precisely the Lie algebras {\em of coclass one}.
Clearly any Lie algebra of maximal class can be generated by two elements.
We will only be interested in infinite-dimensional Lie algebras in this paper.

Now let $L$ be an infinite-dimensional Lie algebra of maximal class which is graded over the positive integers, that is,
$L=\bigoplus_{i=1}^{\infty}L_i$ and $[L_i,L_j]\subseteq L_{i+j}$ for all $i,j$.
We will use the notation $L_{(i)}=\bigoplus_{j\ge i}L_j$ for the elements of the {\em filtration} of $L$ associated to the grading.
If $\dim L_1=2$ then we say that such a graded Lie algebra of maximal class $L$ is an {\em algebra of type $1$}.
Here {\em type $1$} refers to the fact that, necessarily, such $L$ is generated by $L_1$, because $L_1\cap L^2=\{0\}$ and hence $L_1+L^2=L$.
Furthermore, because
$L^i\subseteq L_{(i)}$ for $i>1$,
and $L^i$ has codimension $i$, we have
$\sum_{j=1}^{i-1}\dim L_j\le i$.
Combined with the fact that $L_1$ generates $L$, this inductively implies $\dim L_i=1$ for all $i>1$.

Algebras of type $1$ arise naturally as the graded Lie algebras associated to the lower central series of an arbitrary Lie algebra of maximal class $M$,
that is, taking $L=\bigoplus_{i=1}^{\infty}M^i/M^{i+1}$ with respect to the Lie product defined in a natural way.
It is easy to see that over a field of characteristic zero all algebras of type $1$ are isomorphic and have an abelian ideal of codimension one.
Shalev showed in~\cite{Sha:max} that this is far from the case in positive characteristic, and such algebras need not even be soluble.
A systematic investigation of algebras of type $1$ started in~\cite{CMN}.
In particular, constructions of new algebras of type $1$ from a given one were described,
which when done repeatedly together with some limiting processes produce,
over an arbitrary field $F$ of positive characteristic, uncountably many pairwise non-isomorphic algebras of type $1$.
A classification of algebras of type $1$ was achieved in~\cite{CN} for $p$ odd, and later~\cite{Ju:maximal} for $p=2$,
in the sense of proving that any algebra of type 1 can be obtained through the procedures described in~\cite{CMN}.
Algebras of type 1 were simply referred to as {\em graded Lie algebras of maximal class} in~\cite{CMN,CN,Ju:maximal}.

An algebra of type $1$ may, equivalently, be defined as a graded Lie algebra $L=\bigoplus_{i=1}^{\infty}L_i$
with $\dim L_1=2$, $\dim L_i=1$ for $i>1$, and $[L_i,L_1]=L_{i+1}$ for all $i$.
This leads us to define an {\em algebra of type $n$,} for $n>1$, as a graded Lie algebra $L=\bigoplus_{i=1}^{\infty}L_i$
with $\dim L_i=0$ for $1<i<n$ and $\dim L_i=1$ otherwise,
such that $[L_i,L_1]=L_{i+1}$ for all $i\ge n$.
This appears to be the most practical definition (at least in the infinite-dimensional case which we consider in this paper),
and clearly implies that $L$ is of maximal class.

We may also define an algebra of type $n$
as a graded Lie algebra of maximal class $L=\bigoplus_{i=1}^{\infty}L_i$
satisfying $\dim L_1=1$, $\dim L_i=1$ for $i\ge n$, and generated by $L_1$ and $L_n$
(whence $\dim L_i=0$ for $1<i<n$).
In fact, because
$L^i\subseteq L_{(i+n-1)}$ for $i>1$,
the maximal class assumption forces equality, whence
$[L_i,L_1]=L_{i+1}$ for all $i\ge n$.
(This extends an observation in~\cite[p.~271]{CVL00} for $n=2$.)
This argument breaks down if we only assume $\dim L_i\le 1$ for $i\ge n$ instead of $\dim L_i=1$,
see Remark~\ref{rem:other_types} for an abundance of counterexamples.

Before continuing our discussion of algebras of type $n$ we pause to introduce some notation
and a useful general tool.
Long Lie products are to be interpreted using the left-normed convention, meaning that $[x,y,z]$ stands for $[[x,y],z]$.
A convenient shorthand is $[x, z^i]$,
where $[x,z^0]=x$ and $[x,z^i]=[x,z^{i-1},z]$ for $i>0$.
We will make frequent use of the {\em generalized Jacobi identity}
\[
[x,[y,z^i]]=\sum_{j=0}^{i}(-1)^j\binom{i}{j}[x,z^j,y,z^{i-j}].
\]

In the rest of this section we recall the sequence of two-step centralizers for algebras of type $1$,
and then suitably extend it to algebras of type $n$.
The terminology of two-step centralizers goes back to early work on $p$-groups of maximal class
started by Norman Blackburn in 1958, see~\cite[Kapitel~III, \S14]{Hup}.
Correspondingly, two-step centralizers for algebras of type $1$ were defined in~\cite{CMN} as the centralizers $C_{L}(L^i/L^{i+2})$.
Because $C_L(L^i/L^{i+2})=C_{L_1}(L_i)+L^2$ they are determined by the one-dimensional subspaces $C_i=C_{L_1}(L_i)$ of $L_1$ for $i>1$,
but the `two-step' qualifier was maintained in all later work on graded Lie algebras of maximal class.
The two-step centralizers can also be described by points on a projective line over $F$ or, depending on a choice of two (homogeneous) generators for $L$,
by elements of the field $F$ plus $\infty$.
After making a choice of generators $y,z\in L_1$ for $L$ with $C_2=F y$, which means $[y,z,y]=0$, the two-step centralizers of $L$ are described by a sequence
$(\alpha_i)_{i>1}$ of elements of $F\cup\{\infty\}$ by $C_i=F\cdot(y-\alpha_iz)$
(interpreting $F\cdot(y-\infty\cdot z)$ as $F z$).
We call that {\em the sequence} of $L$, with the definite article despite the fact that it depends on the choice of generators.

A crucial consequence of our special choice of the generator $y$ is
that at least one of each pair of consecutive entries of the sequence vanishes~\cite[Lemma~3.3]{CMN}.
In other words, each nonzero entry $\alpha_i$ of the sequence is immediately followed by a zero entry.
This is equivalent to $[L_i,y,y]=0$ for all $i\ge 1$.
When $i$ equals $1$ or $2$ this holds because $[L_2,y]=0$
by definition of $y$.
Proceeding inductively,
suppose $[L_i,y]\neq 0$ for some $i>2$,
and assume $[L_{i-1},y]=0$ as we may.
Then letting $u$
span $L_{i-1}$ we have that $[u,z]$ spans $L_i$, and
\begin{equation*}
0=[u,[z,y,y]]=[u,z,y,y]-2[u,y,z,y]+[u,y,y,z]=[u,z,y,y]
\end{equation*}
yields $[L_i,y,y]=0$ as desired.
Incidentally, in characteristic different from two
the condition $[L,y,y]=0$ means that
$y$ is a {\em sandwich element} of $L$.
Some implications of this for algebras of type $1$ and for
{\em thin Lie algebras} as well are discussed in~\cite{Mat:sandwich}.
For the present discussion it means that the sequence of $L$ consists of isolated nonzero entries, separated by blocks of one or more zeroes.
Of course the sequence is not uniquely determined by $L$ but still depends on the choice of the other generator $z$.
It follows from~\cite{CMN} that over any countable field $F$ of positive characteristic there exist algebras of type $1$
such that $\bigcup_{i>1}C_i=L_1$, meaning that every element of $F\cup\{\infty\}$ appears somewhere in the sequence of two-step centralizers.

Following~\cite{CMN} we say that an algebra $L$ of type $1$ is {\em uncovered} if $\bigcup_{i>1}C_i\not=L_1$.
Then choosing $z\in L_1 \setminus \bigcup_{i>1}C_i$ we have $L_i=[L_{i-1},z]$ for all $i \ge 1$.
As in~\cite[Section~3]{ShZe:narrow-Witt}, such an element $z$ will be called {\em uniform}.
Being uncovered is clearly not a strong restriction as it can always be achieved by extending the field $F$ of definition
(and is automatically satisfied if $F$ is uncountable).
With such a choice of $z$, the sequence of $L$ is made of elements of $F$ and does not involve the symbol $\infty$.
For uniformity with algebras of type $n>1$ considered below,
it will be convenient to rename the generator $y$ as $e_1$, and recursively define a basis of $L$ by setting $e_i=[e_{i-1},z]$ for all $i>1$.
With this notation the corresponding sequence $(\alpha_i)_{i>1}$
may also be defined by the equation $[e_i,e_1]=\alpha_{i}e_{i+1}$ for all $i>1$.

A definition of the sequence of an algebra of type $n$, which we introduce now,
will be naturally modelled on that for an uncovered algebra of type $1$.
Thus, consider an algebra $L$ of type $n>1$, generated by $z$ of degree $1$ and $e_n$ of degree $n$.
We recursively set $e_i=[e_{i-1},z]$ for $i>n$.
Hence $L_i=F e_i$ except for $L_1=F z$, and of course $L_i=0$ for $1<i<n$ by definition of type $n>1$.
We define the sequence $(\beta_i)_{i>n}$ of $L$ by
$[e_i, e_n]=\beta_i e_{i+n}$.
This includes the case of uncovered algebras of type $1$ as a special case, where of course $L_1=F z+F e_1$.
This sequence of elements of $F$ encodes the adjoint action of $e_n$, while the action of $z$ is actually used to define the chosen basis elements $e_i$.
It follows that the Lie algebra $L$ is completely determined by this sequence, because the adjoint action of the generators
determines the adjoint action of every element of the Lie algebra.
Note that the equation used to define $\beta_i$ would naturally extend to $\beta_n=0$,
but for various reasons (one of which relates to {\em translation,} see a comment after the proof of Proposition~\ref{prop:translation})
it is best to leave $\beta_n$ undefined and let the sequence start with $\beta_{n+1}$.

The sequence $(\beta_i)_{i>n}$ is the zero sequence for the algebra of type $n$ having multiplication
$[e_i,e_n]=0$ for all $i>n$ (and $e_i:=[e_n,z^{i-n}]$ as postulated above), which has an abelian maximal ideal,
spanned by $e_1$ and $L^2=L_{(n+1)}$.
When $n=1$ this is, up to isomorphism,
the only metabelian algebra of type $n$,
but when $n>1$ there is another one.
In fact, if $L$ is metabelian of type $n$
then because $[e_n,z]=e_{n+1}$
belongs to the abelian ideal $L^2$ we find
\[
0=[e_i,[e_n,z]]=[e_i,e_n,z]-[e_i,z,e_n]
=(\beta_i-\beta_{i+1})e_{i+n+1}
\]
for $i>n$, and hence the sequence $(\beta_i)_{i>n}$ is constant.
If this constant $\beta$ is not zero then after renaming
$(1/\beta)e_n$ as $e_n$ we find
$[e_i,e_n]=e_{i+n}$ for $i>n$.
This specification, together with $[e_i,e_j]=0$ for $i,j>n+1$,
clearly defines a Lie algebra.
When $n>1$ the maximal abelian ideal of this algebra $L$
is $L^2=L_{(n+1)}$, of codimension two,
and hence this algebra is not isomorphic to the one considered
above where $(\beta_i)_{i>n}$ was the zero sequence.
However, when $n=1$ this algebra is isomorphic to the previous
one, as we see by renaming $e_1-z$ as $e_1$.

Aside from the algebra of type $n$ with an abelian maximal ideal, the definition of the sequence of $L$ clearly depends on the choice of the generators $z$ and $e_n$.
For algebras of type $n>1$ those generators (if assumed homogeneous of degrees $1$ and $n$) are unique up to scalar multiples.
If $z'=\lambda z$ and $e_n'=\mu e_n$ are another pair of generators, for some $\lambda, \mu \in F^*$, then for all $i>n$ we inductively find
$e_i'=[e_{i-1}',z']=\lambda^{i-n}\mu e_i$, and hence
\[
[e_i',e_n']=\lambda^{i-n} \mu^2[e_i,e_n]=\mu/\lambda^n \beta_i e_{i+n}',
\]
which means $\beta_i'=\mu/\lambda^{n} \beta_i$.
Thus, multiplying the sequence of an algebra of type $n$ by a non-zero scalar
does not affect the isomorphism type of the algebra, and this is the only change to the sequence which can be obtained through a change of generators.
It will be convenient to work, without further mention, with a \emph{normalized sequence} whose first nonzero entry equals $1$.
Therefore, two algebras of type $n>1$ are isomorphic if and only if their normalized sequences coincide.

For comparison, algebras of type $1$ allow greater variation in choosing the generating elements $e_1$ and $z$,
as those have the same degree, with the only constraints that
$[e_2,e_1]=0$ (having chosen $e_1$ to span the first two-step centralizer $C_2$, as we may),
and that $z$ is uniform, thus $e_i=[e_1,z^i]$ spans $L_i$ for $i>1$.
In particular, any other pair of generators satisfying those constraints has the form
$z'=\lambda z+\nu e_1$ and $e_1'=\mu e_1$, for certain $\lambda, \mu \in F^*$ and $\nu\in F$.
Then $e_{i+1}'=(\lambda+\nu\alpha_i)e_i'$ for $i>1$, and hence $\alpha_i'=\mu\alpha_i/(\lambda+\nu\alpha_i)$ for $i\ge 1$.
A natural normalization of the sequence $(\alpha_i)_{i>1}$ for algebras of type $1$
is assuming that, unless $L$ is metabelian, its {\em first constituent} (as defined in the next section) ends with a $1$.
This means that the second two-step centralizer $C_{\ell}$ is spanned by $x=z-e_1$, where $y=e_1$ spans the first by earlier assumption.
Such generators $x$ and $y$ of an algebra of type $1$ were the standard choice of generators adopted in~\cite{CMN,CN,Ju:maximal}.
As described in~\cite[p.~4025]{CMN}, if $L$ has a third two-step centralizer (which corresponds to the earliest occurrence of an entry different from $0$ and $1$
in the sequence $(\alpha_i)_{i>1}$), then one may assume that to take a third pre-assigned value
and thus make the sequence $(\alpha_i)_{i>1}$ uniquely determined by the algebra.

\begin{rem}\label{rem:other_types}
Consider an algebra $L$ of type 1 having only two distinct two-step centralizers,
namely $C_2=Fy$ and $C_{\ell}=Fx$ according to the naming conventions recalled above.
Here $\ell$ is the {\em length of the first constituent,} which we discuss in the next section,
and is known from~\cite{CMN,Ju:maximal} to equal twice a power of the characteristic for algebras of type $1$.
Then $L$ actually admits a {\em double grading,} meaning over $\Z^2$, where the generators $x$ and $y$ are assigned degrees $(1,0)$ and $(0,1)$, respectively.
According to the classification in~\cite{CMN} (and~\cite{Ju:maximal} in characteristic two) the class of such algebras exhibits almost the same complexity
as without this restriction on the two-step centralizers, and they similarly form $|F|^{\aleph_0}$ isomorphism classes.
By assigning new degrees $m$ and $n$ to $x$ and $y$, where $m$ and $n$ are arbitrary integers, we get a new $\Z$-grading of $L$
where the component of $\Z^2$-degree $(i,j)$ gets degree $im+jn$.
If $0<m<n$ then all homogeneous components of this new $\Z$-grading have dimension at most one.
In particular, when $m=1$ we obtain a huge variety of graded Lie algebras of maximal class generated by an element of degree 1 and one of degree $n$,
which, however, are not algebras of type $n$ according to our definition, because the component of degree $n+\ell$ is zero.
\end{rem}

We conclude this section with a brief discussion of algebras of type $2$ in characteristic zero, with attention to their sequences as introduced
in this section.
According to~\cite[Theorem~7.1]{ShZe:narrow-Witt}, over a field $F$ of characteristic zero there are exactly three algebras of type $2$ up to isomorphism.
Two of them are metabelian, with sequences given by $\beta_i=0$ for $i>1$, and $\beta_i=1$ for $i>1$, respectively.
(We have discussed the two metabelian algebras of arbitrary type $n>1$ earlier in this section.)
The third algebra, denoted by $W$, is insoluble, and is the {\em positive} part of the Witt algebra $\Der(F[X])$.
Hence $W$ has a graded basis given by
$E_i=x^{i+1}\partial/\partial x$
for $i>0$, with Lie product
$[E_i,E_j]=(j-i)E_{i+j}$.
To relate to the notation employed in this section, setting $z=-E_1$ and $e_2=-E_2$ we find
$e_i:=[e_2,z^{i-2}]=-(i-2)!\,E_i$
for $i\ge 2$, and hence
\[
[e_i,e_2]
=(i-2)!\,[E_i,E_2]
=(i-2)!\,(2-i)\,E_{i+2}
=\frac{i-2}{i(i-1)}e_{i+2}.
\]
If we insist on normalizing the sequence with $\beta_3=1$, then we may replace $e_2'=-6E_2$, so that
$W$ has sequence $(\beta_i)_{i>2}$ given by $\beta_i=6(i-2)/(i^2-i)$,
which starts as $1,1,9/10,\ldots$ (see also~\cite[Section~7]{ShZe:narrow-Witt}).

Note that if we read the Lie product in $W$, with respect to its basis given by the elements $E_i$, modulo a prime $p$,
we do not get a Lie algebra of maximal class.
Thus, this characteristic $p$ version of $W$, say over the prime field $\F_p$, hence the positive part of $\Der(\F_p[X])$,
has no place in the classification of algebras of type $2$ in~\cite{CVL00,CVL03}.

\section{Constituents of algebras of type $n$}\label{sec:constituents}

Consider an uncovered algebra $L$ of type $1$.
If $L$ is metabelian then $L$ has the zero sequence.
Otherwise, we saw in the previous section that the sequence of $L$ consists of isolated occurrences of nonzero scalars $\alpha_j$,
separated by strings of one or more zeroes.
This suggests a natural partition of the sequence into juxtaposed blocks, each made of zeroes except for a final nonzero entry.
Those blocks (or finite subsequences) are called the {\em constituents} of $L$,
and will be numbered in order of occurrence (the {\em first} constituent, the {\em second,} etc.).
The {\em length} of a constituent is simply the number of its entries, except for the first constituent where that number needs to be increased by one
(due to having defined the two-step centralizers only starting from $C_2$).
This definition of constituents can be suitably rephrased to apply to algebras of type $1$ which are not uncovered,
and as such it was given in~\cite{CN} (slightly adjusted from the original one in~\cite{CMN}).

Constituents and their lengths were crucial in the investigation of algebras of type $1$ started in~\cite{CMN},
and in their subsequent classification in~\cite{CN,Ju:maximal}.
In particular, it was shown in~\cite{CMN} that the first constituent of a non-metabelian algebra of type $1$ has length $2q$ for some (positive) power $q$ of $p$,
called the {\em parameter} of the algebra.
It was also shown that the possible lengths of other constituents are either $2q$ or $2q-p^\beta$,
with $p^\beta\ge 1$ a power of $p$ not exceeding the parameter $q$.
Some of those results were revisited in~\cite{Mat:chain_lengths} with a different approach, which has inspired our polynomial method of Section~\ref{sec:polynomials}.
In particular, if the first constituent of $L$
has length $\ell=2q$, then
the length $\ell_r$ of each constituent
satisfies $\ell/2 \le \ell_r\le \ell$.
In Lemmas~\ref{lemma:upper_bound} and~\ref{lemma:constituent_bound}
we will extend this fact to constituents of
algebras of type $n$, as defined below.

There have been inconsistencies in the definition of constituents for algebras of type $n>1$.
Our definition will extend the one given in~\cite{CVL03} for algebras of type $2$.
In order to motivate the definition we first describe a natural way to obtain an algebra of type $n$ from any uncovered algebra of type $1$.

Let $N=\bigoplus_{i \ge 1} N_i$ be an uncovered algebra of type $1$.
Fix $n>1$ and consider the (graded) subalgebra $L$ of $N$  generated by $z$ and $e_n$.
Hence $L=L_1\oplus \bigoplus_{i \ge n} L_i$ is an algebra of type $n$,
where $L_1=F z$, and $L_i=N_i$ for $i\ge n$.
The (right) adjoint action of $e_n$ on $L$ is described by
\begin{align*}
[e_i,e_n]
&=[e_i,[y,z^{n-1}]]
=\sum_{k=0}^{n-1} (-1)^k \binom{n-1}{k}[e_{i+k},y,z^{n-1-k}]\\
&=\sum_{k=0}^{n-1} (-1)^k \binom{n-1}{k}\alpha_{i+k} e_{i+n}.
\end{align*}
Hence the sequence $(\beta_i)_{i>n}$ of $L$ is given by
$\beta_i=\sum_{k=0}^{n-1} (-1)^k \binom{n-1}{k}\alpha_{i+k}$.
Note that different choices of $z$ may lead to non-isomorphic subalgebras $L$.

As recalled earlier the sequence of $N$ consists of isolated occurrences of nonzero scalars $\alpha_j$,
separated by (finite) strings of one or more zeroes.
If a nonzero $\alpha_j$ is preceded by at least $n-1$ zeroes and followed by at least $n-1$ zeroes, then
$\beta_i=(-1)^{j-i}\binom{n-1}{j-i}\alpha_j$ for $j-n<i\le j$.
Letting $J=\{j:\alpha_j\neq 0\}$, the sequence of $N$ then consists of adjacent blocks
$(0,\ldots,0,\alpha_j)$,
which by definition are the constituents of $N$.

Now assume every constituent of $N$
has length at least $n$.
Then we have just discovered that
the sequence of $L$ consists of adjacent blocks
$(\beta_{j-m+1},\ldots,\beta_{j})$
of the form
\begin{equation}\label{eq:ordinary_from_type_1}
\beta_{j-i}
=
(-1)^i\binom{n-1}{i}\beta_j
\qquad\text{for $0\le i<m$.}
\end{equation}
where $m$ is the length of the constituent of $N$ ending in $\alpha_j$,
or $n$ less than that length in case of the first constituent.
Note that the first nonzero entry of this block is $\beta_{j-n+1}=(-1)^{n-1}\alpha_j$.
It seems natural to call those blocks $(\beta_{j-m+1},\ldots,\beta_{j})$ the {\em constituents} of $L$.
Their lengths coincide with the lengths of the constituents of the overalgebra $N$ of type $1$,
if we stipulate that the length of the first constituent of $N$ (which starts with $\beta_{n+1}$) should,
by definition, exceed by $n$ its actual length as a sequence.

Note that if $N$ of type $1$
has first constituent of length $\ell=2q$,
then as we recalled above from~\cite{CMN}
each other constituent has length
at least $\ell/2$, and so a subalgebra $N$ of type $n$
constructed as above will have  constituents
of the form given by Equation~\eqref{eq:ordinary_from_type_1}
as long as $\ell\ge 2n$.
However, in absence of this condition some
$\beta_i=\sum_{k=0}^{n-1} (-1)^k \binom{n-1}{k}\alpha_{i+k}$
may depend on two or more nonzero entries $\alpha_j$
of the sequence of $L$.

For an arbitrary algebra $L$ of type $n>1$, we will also divide its sequence $(\beta_i)_{i>n}$
into adjacent blocks, and call those its {\em constituents}, but the entries of a constituent need not generally follow the special pattern described above.
Assume $(\beta_i)_{i>n}$ is not the zero sequence.
Then we define the {\em first constituent} of $L$ as
$(\beta_{n+1},\ldots,\beta_{\ell})$,
where $\beta_{\ell-n+1}$ is the first nonzero entry of the sequence $(\beta_i)_{i>n}$.
Thus, $\ell$ is determined by $[e_{\ell-n+1},e_n]\neq 0$ but $[e_i, e_n]=0$ for $n\le i\le\ell-n$.
We stipulate that the first constituent has length $\ell$ (despite actually having only $\ell-n$ entries as a sequence).

Note that our definition implies $\ell\ge 2n$.
In particular, if $L$ is obtained as
a subalgebra of type $n$ of an algebra $N$ of type $1$
according to the construction given earlier,
and $N$ has first constituent shorter than $2n$,
then $L$ and $N$ will have
{\em different} first constituent lengths.

We define further constituents of
an arbitrary algebra $L$ of type $n>1$ recursively.
If $(\beta_i,\ldots,\beta_j)$ is a constituent which we have already defined,
and $\beta_{j+m-n+1}$ is the first nonzero entry of the sequence past $\beta_j$
(which exists according to Lemma~\ref{lemma:upper_bound} below), then
we define $(\beta_{j+1},\ldots,\beta_{j+m})$ to be the next constituent, of length $m$.
Thus, the sequence of $L$ gets partitioned into a sequence of adjacent constituents.
Note that, by definition, every constituent has length at least $n$, and the first constituent has length at least $2n$.
We denote the constituent lengths by $\ell=\ell_1,\ell_2,\ell_3,\ldots$, numbered in order of occurrence.

We have seen above that if an algebra $L$ of type $n$
is a subalgebra of an algebra $N$ of type $1$,
whose first constituent has length at least $2n$,
then the constituents of $L$ have the special form
described by Equation~\eqref{eq:ordinary_from_type_1}.
The next result shows that the converse holds.
It is a stronger version of a result
of Ugolini~\cite[Lemma~3.1.2]{Ugo:thesis}, which in turn extended the special case $n=2$ proved in~\cite[Section~2]{CVL00}
(also stated as~\cite[Proposition~2.2]{CVL03}).
Ugolini's version required an assumption that the constituents are not too short, see Remark~\ref{rem:Ugolini} below for details.

\begin{prop}[Ugolini]\label{lemma:Ugolini}
Let $L$ be an algebra of type $n$ and
suppose that each constituent of $L$ has the special form
$(\lambda_{j-m+1},\ldots,\lambda_j)$, with
\[
\lambda_{j-i}
=
(-1)^i\binom{n-1}{i}\lambda_j
\qquad\text{for $0\le i<m$,}
\]
for some integer $m\ge n$.
Then $L$ is a subalgebra of an uncovered algebra of type $1$.
\end{prop}

\begin{proof}
The proof is by induction on $n$, with $n=1$ as the trivially true base case.
Thus, let $L$ be an algebra of type $n$, for some $n>1$, satisfying the stated hypotheses.
It will be sufficient to embed $L$ as a (graded) subalgebra of an algebra $\bar L$ of type $n-1$
satisfying the corresponding hypotheses.

Let $J$ be the set of degrees of the final element of each constituent of $L$.
Order the elements of $J$ as $j_1<j_2<\cdots$ and set $j_0=0$ for convenience, so that the $r$th constituent has length $\ell_r=j_r-j_{r-1}$ for all $r$.
Then $L$ is generated by $z$ and $e_n$, with relations
\begin{equation}\label{eq:relations}
[e_i,e_n]=
(-1)^{j_r-i} \binom{n-1}{j_r-i}\lambda_{j_r} e_{i+n}
\qquad
\text{for $j_{r-1}<i\le j_r$},
\end{equation}
where of course we require that $i>n$ as well.
We will extend $L$ to an algebra $\bar L$ of type $n-1$ containing $L$ as an ideal of codimension one,
and spanned over $L$ by an element $e_{n-1}$ of degree $n-1$,
satisfying
\begin{equation}\label{eq:relations-}
[e_i,e_{n-1}]=
(-1)^{j_r-i} \binom{n-2}{j_r-i}\lambda_{j_r} e_{i+n-1}
\qquad
\text{for $j_{r-1}<i\le j_r$}.
\end{equation}
To achieve that we view $L$ as a quotient of a free Lie algebra $\mathcal{L}(X)$ on a free generating set $X=\{z,e_n\}$,
with relations given by Equation~\eqref{eq:relations}.
Hence $L=\mathcal{L}(X)/I$, where $I$ is the ideal of $\mathcal{L}(X)$ generated by the relators
\begin{equation}\label{eq:relators}
[e_i,e_n]-
(-1)^{j_r-i} \binom{n-1}{j_r-i}\lambda_{j_r} e_{i+n}
\qquad
\text{for $j_{r-1}<i\le j_r$}.
\end{equation}

By general properties of free Lie algebras there is a unique derivation $D$ of $\mathcal{L}(X)$ such that
$Dz=e_n$ and $De_n=0$.
We now prove inductively that
\begin{equation*}
De_i=
(-1)^{j_r-i+1} \binom{n-2}{j_r-i}\lambda_{j_r} e_{i+n-1}
\qquad
\text{for $j_{r-1}<i\le j_r$}.
\end{equation*}
In fact, for $j_{r-1}+1<i\le j_r$, starting with $i=n+1$ in case $r=1$, we have
\begin{align*}
De_i
&=
D[e_{i-1},z]
=[De_{i-1},z]+[e_{i-1},e_n]
\\&=
(-1)^{j_r-i} \binom{n-2}{j_r-i+1}\lambda_{j_r} e_{i+n-1}
+(-1)^{j_r-i+1} \binom{n-1}{j_r-i+1}\lambda_{j_r} e_{i+n-1}
\\&=
(-1)^{j_r-i+1} \binom{n-2}{j_r-i}\lambda_{j_r} e_{i+n-1}.
\end{align*}
In the remaining cases, where $i=j_r+1$ for some $r$, we have
\begin{align*}
De_{j_r+1}
&=
D[e_{j_r},z]
=[De_{j_r},z]+[e_{j_r},e_n]
\\&=
-\lambda_{j_r} e_{j_r+n}
+\lambda_{j_r} e_{j_r+n}
\\&=
(-1)^{j_{r+1}-j_r} \binom{n-2}{j_{r+1}-j_r-1}\lambda_{j_{r+1}} e_{j_r+1},
=0
\end{align*}
because $j_{r+1}-j_r=\ell_{r+1}$, a constituent length, cannot be less than $n$ by definition.

Now we show that $DI\subseteq I$, which we need for $D$ to induce a derivation of the quotient $\mathcal{L}(X)/I$.
In fact we will show that $DI=0$.
It suffices to show that each relator in Equation~\eqref{eq:relators} belongs to the kernel of $D$.
On the one hand, for $j_{r-1}<i\le j_r$ we have
\[
D[e_i,e_n]
=[De_i,e_n]
=(-1)^{j_r-i+1} \binom{n-2}{j_r-i}\lambda_{j_r} [e_{i+n-1},e_n].
\]
The binomial coefficient vanishes unless $i+n-1>j_r$, in which case
\[
D[e_i,e_n]
=
(-1)^{j_{r+1}-j_r-n} \binom{n-2}{j_r-i}\lambda_{j_r}
\binom{n-1}{j_{r+1}-i-n+1}\lambda_{j_{r+1}}
e_{i+2n-1}.
\]
On the other hand, for $j_{r-1}<i\le j_r$ the binomial coefficient in
\[
(-1)^{j_r-i} \binom{n-1}{j_r-i}\lambda_{j_r} De_{i+n}
\]
vanishes unless $i+n>j_r$, in which case this expression equals
\[
(-1)^{j_{r+1}-j_r-n+1} \binom{n-1}{j_r-i}\lambda_{j_r}
\binom{n-2}{j_{r+1}-i-n}\lambda_{j_{r+1}}
e_{i+2n-1}.
\]
However, note that when $i=j_r-n+1$ the second binomial coefficient in this expression vanishes because $j_{r+1}-j_r\ge n$.
Because for $i+n-1>j_r$ we have
\[
\binom{n-1}{j_r-i}=
\frac{n-1}{n-1-j_r+i}\binom{n-2}{j_r-i}
\]
and
\[
\binom{n-1}{j_{r+1}-i-n+1}=
\frac{n-1}{j_{r+1}-i-n+1}\binom{n-2}{j_{r+1}-i-n}
\]
the two expressions we have found are equal.
We conclude that every relator~\eqref{eq:relators} belongs to the kernel of $D$.
Consequently, $D$ induces a derivation of $\mathcal{L}(X)/I$, which we have identified with $L$.
Taking as $\bar L$ the semidirect product of $L$ with a one-dimensional Lie algebras spanned by the derivation $D$,
and calling $e_{n-1}$ the element of $\bar L$ corresponding to $D$, we find that $\bar L$ is an algebra of type $n-1$,
satisfying the relations in Equation~\eqref{eq:relations-} as desired.
\end{proof}

\begin{rem}\label{rem:Ugolini}
Note that in the final part of the above proof
each of the two terms of the relator in Equation~\eqref{eq:relators} vanishes separately, rather than just their difference,
except in the special situation where $j_r-j_{r-1}=j_{r-1}-j_r=n$ and $i=j_{r-1}+1$.
That cannot occur, in particular, if each constituent length $\ell_r$ exceeds its minimal value $n$ allowed by the definition.
Such assumption was amply satisfied in Ugolini's proof of his version of this result~\cite[Lemma~3.1.2]{Ugo:thesis}
because of his blanket assumption $p>2n$, and that made his proof a little easier than ours.
\end{rem}

In conclusion, if $N$ is an algebra
of type $1$, with first constituent of length $2q$,
and an algebra $L$ of type $n$ with $n\le q$
is obtained as a subalgebra of $N$
according to the construction given earlier
in this section,
then according to Lemma~\ref{lemma:Ugolini}
this origin of $L$ can be detected from its
sequence $(\beta_i)_{i>n}$.
Furthermore, $N$ can be uniquely recovered from $L$,
because the sequence $(\alpha_i)_{i>1}$ of $L$ can be recovered from the sequence of $N$.
No such claim is made if $n>q$.
In fact, as an extreme example, if $N$ is soluble
then a subalgebra $L$ of type $n$ will be metabelian
for $n$ large enough.

\section{Constituent lengths}\label{sec:const_lengths}

In this section we establish some general facts on the constituent lengths $\ell=\ell_1,\ell_2,\ell_3,\ldots$ of an algebra $L$ of arbitrary type $n$.
Assuming $L$ not metabelian, we show that $\ell$ is even, and that all constituent lengths satisfy $\ell/2\le\ell_r\le\ell$.
The upper bound on $\ell_r$ guarantees, in particular, that all constituent lengths are finite,
which we assumed when defining them in Section~\ref{sec:constituents}.
When $n=p$ the lower bound on $\ell_2$ will be crucial for our proof of Theorem~\ref{thm:first_length}.
The lower bound for further constituent lengths $\ell_r$ will not be used in this paper,
but is needed in other parts of the proof of Theorem~\ref{thm:classification} as presented in~\cite{Sca:thesis}.
However, it takes no more effort to prove those bounds for arbitrary constituent lengths, and for algebras of arbitrary type $n$.

For convenience we recall our general notation for algebras of type $n$ from Section~\ref{sec:constituents}.
Thus, let $L$ be an algebra of type $n$, generated by two elements $z$ and $e_n$, of degrees $1$ and $n$.
Hence $L_i=Fe_i$ for each $i>n$, where $e_i=[e_n,z^{i-n}]$.
Let $(\beta_i)_{i>n}$ be the sequence of $L$, where $[e_i,e_n]=\beta_ie_{i+n}$,
with the normalization $\beta_{\ell-n+1}=1$.
According to the definition of constituents we have
\begin{subequations}
\begin{align}
[e_i, e_n]&=0\quad\text{for }n<i\le\ell-n,
 \label{eq:first_const_n}\\
[e_{\ell-n+1},e_n]&=e_{\ell+1},
 \label{eq:first_const_end_n}\\
[e_{\ell+j},e_n]&=0\quad\text{ for }0<j\le \ell_2-n,
 \label{eq:second_const_n}\\
[e_{\ell+\ell_2-n+1},e_n]&=\beta_{\ell+\ell_2-n+1}e_{\ell+\ell_2+1}\ne 0,
 \label{eq:second_const_end_n}
\end{align}
\end{subequations}
with equations analogous to Equations~\eqref{eq:second_const_n} and~\eqref{eq:second_const_end_n} for further constituents.

\begin{lemma}\label{lemma:ell_even}
The length $\ell$ of the first constituent of a non-metabelian algebra $L$ of type $n$ is even.
\end{lemma}

\begin{proof}
For algebras of type $1$ this is part of~\cite[Lemma~5.1]{CMN},
and as mentioned there it reflects a well-known property of exceptional $p$-groups of maximal class.
The proof for $n>1$ is the same aside from notation.
Thus, if $\ell$ is odd then $\ell-2n+1$ is even, say $\ell-2n+1=2r$, and we have
\[
0=[[e_n,z^r],[e_n,z^r]]
=\sum_{k=0}^{r} (-1)^k \binom{r}{k}[e_n,z^{r+k},e_n,z^{r-k}]
=(-1)^r [e_n,z^{2r},e_n],
\]
because $\beta_i=0$ for $i\le\ell-n$.
However, this contradicts the fact that $\beta_{\ell-n+1}\neq 0$.
\end{proof}

The calculation in the above proof, when done with an arbitrary $r$,
\[
0=[[e_n,z^r],[e_n,z^r]]
=\sum_{k=0}^{r} (-1)^k \binom{r}{k}[e_n,z^{r+k},e_n,z^{r-k}]
\]
yields
\[
\sum_{k=0}^{r} (-1)^k \binom{r}{k}\beta_{n+r+k}
=0,
\]
allowing us to express $\beta_{n+2r}$ as a linear combination of earlier entries of the sequence.
Thus, any $\beta_j$ where $j$ has the same parity as $n$ can be expressed as a linear combination of the previous ones.
One may view Lemma~\ref{lemma:ell_even} as a direct consequence of this simple fact, which for $n=2$ appeared as~\cite[Lemma~2.1]{CVL03}.
Its special case where $r=\ell/2-n+1$ yields
\begin{equation}\label{eq:beta}
\beta_{\ell-n+2}=(\ell/2-n+1)\beta_{\ell-n+1},
\end{equation}
which we will use near the end of our proof of Theorem~\ref{thm:first_length}, in Section~\ref{sec:first_length_proof}.

\begin{lemma}\label{lemma:upper_bound}
Let $L$ be a non-metabelian algebra of type $n$, with first constituent of length $\ell$.
Then its sequence $(\beta_i)_{i>n}$ can contain at most $\ell-n$ consecutive zero entries.
Consequently, the constituents lengths $\ell=\ell_1,\ell_2,\ell_3,\ldots$ of $L$ satisfy $\ell_r\le\ell$ for all $r$.
\end{lemma}

\begin{proof}
Suppose for a contradiction that a nonzero entry $\beta_i$ of the sequence is followed by at least $\ell-n+1$ zeroes,
that is, $[L_{i+j},e_n]=0$ for $0<j\le\ell-n+1$.
Using the generalized Jacobi identity we find
\begin{align*}
[e_i,e_{\ell+1}]
=[e_i,[e_n,z^{\ell-n+1}]]
&=\sum_{j=0}^{\ell-n+1} (-1)^j \binom{\ell-n+1}{j}[e_{i+j},e_n,z^{\ell-n+1-j}]\\
&=[e_i,e_n,z^{\ell-n+1}]
=\beta_i e_{i+\ell+1}
\neq 0.
\end{align*}
We now expand $[e_i,e_{\ell+1}]$ in a different way, using
$[e_{\ell-n+1},e_n]=e_{\ell+1}$,
and we find
\[
[e_i,e_{\ell+1}]
=[e_i,[e_{\ell-n+1},e_n]]
=[e_i,e_{\ell-n+1},e_n]-[e_i,e_n,e_{\ell-n+1}].
\]
Each of the two Lie products in the difference vanishes because of our assumption,
and this provides the desired contradiction.
\end{proof}

Note that the proof of Lemma~\ref{lemma:upper_bound}
may be interpreted as establishing the weaker conclusion that
the sequence of $L$ cannot be zero from some point on,
thus justifying our recursive definition of constituents in Section~\ref{sec:constituents}.

\begin{lemma}\label{lemma:constituent_bound}
Let $L$ be a non-metabelian algebra of type $n$, with constituent lengths $\ell=\ell_1,\ell_2,\ell_3,\ldots$ in order of occurrence.
Then $\ell_r\ge \ell/2$ for all $r$.
\end{lemma}

\begin{proof}
We start with proving the claim for the second constituent.
Together with the generalized Jacobi identity, Equation~\eqref{eq:first_const_n} yields
$[e_i,e_j]=[e_i,[e_n,z^{j-n}]]=0$
for $i,j\ge n$ with $i+j\le\ell$.
In particular, for $n\le j <\ell/2$ we have
$[e_{j+1},e_{j}]=0$
and $[e_{\ell-j},e_j]=0$, but $[e_{\ell-j},e_{j+1}]=[e_{\ell-j},[e_n,z^{j+1-n}]]=(-1)^{j+1-n}e_{\ell+1}$
according to Equation~\eqref{eq:first_const_end_n}, whence
\[
0=[e_{\ell-j},[e_{j+1},e_{j}]]
=[e_{\ell-j},e_{j+1},e_{j}]
=(-1)^{j+1-n}[e_{\ell+1},e_{j}].
\]
However, Equations~\eqref{eq:second_const_n} and~\eqref{eq:second_const_end_n} together with the generalized Jacobi identity imply
\[
[e_{\ell+1},e_{\ell_2}]=(-1)^{\ell_2-n}[e_{\ell+\ell_2-n+1},e_n]\neq 0,
\]
whence $\ell_2\ge\ell/2$ as desired.

The proof for subsequent constituents is analogous, proceeding by induction.
Thus, suppose we have shown that $\ell_r\ge\ell/2$ for some $r>1$.
Setting $s=\ell+\ell_2+\cdots+\ell_r$,
the $r$-th constituent is characterized by
$[e_{s-\ell_r+j},e_n]=0$ for $0<j\le\ell_r-n$
and $[e_{s-n+1},e_n]=\beta_{s-n+1}e_{s+1}\neq 0$.
As in the base case $r=2$, for $n\le j <\ell/2$ these equations together with Equation~\eqref{eq:first_const_n} and~\eqref{eq:first_const_end_n} imply
\[
0=[e_{s-j},[e_{j+1},e_{j}]]
=[e_{s-j},e_{j+1},e_{j}]
=(-1)^{j+1-n}\beta_{s-n+1}[e_{s+1},e_{j}],
\]
whence $[e_{s+1},e_{j}]=0$.
However, the corresponding equations for the next constituent, which are
$[e_{s+j},e_n]=0$ for $0<j\le\ell_{r+1}-n$
and $[e_{s+\ell_{r+1}-n+1},e_n]\ne 0$, imply
\[
[e_{s+1},e_{\ell_{r+1}}]=(-1)^{\ell_r-n}[e_{s+\ell_{r+1}-n+1},e_n]\neq 0.
\]
We conclude $\ell_{r+1}\ge\ell/2$ as desired.
\end{proof}

\section{A characterization of the constituent lengths}\label{sec:characterization}

In this section we show how the constituent lengths have an intrinsic meaning in terms of $L$ as a Lie algebra alone, with no reference to its grading
(with an exception for the first constituent length requiring the type $n$ of the grading).
In an algebra $L$ of type $n$ we have $L^i=L_{(i+n-1)}$ for $i>1$, and this is the only ideal of $L$ having codimension $i$.
In particular, we have $L^2=L_{(n+1)}$.
Because of the grading we have $[L^2,L_n]=[L_{(n+1)},L_n]\subseteq L_{(2n+1)}=L^{n+2}$,
and by definition the first constituent length is the largest integer $\ell$ such that
$[L^2,L_n]\subseteq L_{(\ell+1)}=L^{\ell-n+2}$.
When $\ell$ exceeds its smallest allowed value $2n$, which means $[L^2,L_n]\subseteq L^{n+3}$
(whence $[L^2,L_n]\subseteq L^{n+4}$ because $\ell$ is even, as we know from Lemma~\ref{lemma:ell_even}),
or equivalently $\beta_{n+1}=0$,
the constituent lengths admit the following internal characterization in terms of the Lie powers of $L^2$.

\begin{prop}\label{prop:constituent_lengths}
Let $L$ be an algebra of type $n>1$ such that $[L^2,L_n]\subseteq L^{n+3}$.
Then the constituent lengths of $L$ are given by
$\ell=\ell_1=\dim\bigl(L^2/(L^2)^2\bigr)+n$
and
$\ell_r=\dim\bigl((L^2)^r/(L^2)^{r+1}\bigr)$ for $r>1$.
\end{prop}

Because $L^2=L_{(n+1)}$, for any algebra of type $n$ we have $(L^2)^2\subseteq L_{(2n+3)}=L^{n+4}$,
as no nonzero homogeneous element of $(L^2)^2$
can have degree lower than the formal degree of $[e_{n+2},e_{n+1}]$.
Consequently, $(L^2)^2$ has codimension at least $n+2$ in $L^2$.
This shows that the characterization of $\ell$ given in Proposition~\ref{prop:constituent_lengths} would necessarily fail without assuming $\ell>2n$ as we did.

\begin{proof}
Because the Lie powers of $L^2=L_{(n+1)}$ are ideals of $L$,
and because $L$ has precisely one ideal of each codimension exceeding $1$, the desired conclusion is equivalent to
$(L^2)^{r+1}=L_{(\ell+\ell_2+\cdots+\ell_r+1)}$
for $r\ge 1$.
While the constituents are defined by the adjoint action of $e_n$,
the Lie powers $(L^2)^r$ are essentially determined by the adjoint action of $e_{n+1}$ on $L^2$, hence we need to analyze that.
Because
\begin{equation}\label{eq:ad_e_{n+1}}
[e_i,e_{n+1}]=[e_i,e_n,z]-[e_{i+1},e_n]
=(\beta_i-\beta_{i+1})e_{i+n+1},
\end{equation}
according to our definition of the first constituent the earliest nonzero such Lie product occurs for $i=\ell-n$, and spans $L_{\ell+1}$.
Because $\ell>2n$ according to our hypothesis
$[L^2,L_n]\subseteq L^{n+3}$
(and then actually $\ell\ge 2n+2$ as $\ell$ is even according to Lemma~\ref{lemma:ell_even}),
such nonzero Lie product $[e_{\ell-n},e_{n+1}]$
belongs to $(L^2)^2$,
and hence $L_{(\ell+1)}\subseteq (L^2)^2$.
To show that equality holds it remains to check that the Lie products
\begin{equation}\label{eq:ad_e_{n+j}}
[e_i,e_{n+j}]
=\sum_{k=0}^{j}(-1)^k\binom{j}{k}[e_{i+k},e_n,z^{j-k}]
=\biggl(\sum_{k=0}^{j} (-1)^k \binom{j}{k}\beta_{i+k}\biggr) e_{i+n+j}
\end{equation}
belong to $L_{(\ell+1)}$ for $i>n$ and $j>0$.
This is so because the coefficients $\beta_{i+k}$ involved vanish unless $i+j\ge\ell-n+1$, in which case
$e_{i+n+j}\in L_{(\ell+1)}$.

To proceed further, note that our hypothesis $\beta_{n+1}=0$ entails $[e_{\ell+1},e_n]=0$, because
\[
0=[e_{\ell-n},[e_n,z,e_n]]=-[e_{\ell-n+1},e_n,e_n]=-\beta_{\ell-n+1}[e_{\ell+1},e_n].
\]
More generally, whenever $\beta_{k-n} =0$ and $\beta_{k-n+1} \ne 0$ for some $k$, a similar calculations yields $\beta_{k+1} =0$.
In particular, this calculation
shows that under the hypothesis
$\beta_{k-n} =0$ every constituent of $L$ begins with a zero entry.
Thus, we have $0=\beta_{\ell+1}=\beta_{\ell+\ell_2+1}$, etc.
Equivalently, every constituent length exceeds $n$.

Now, the length of the second constituent is the smallest integer $\ell_2\ge n$ such that $[e_{\ell+\ell_2-n+1},e_n]\neq 0$.
According to Equation~\eqref{eq:ad_e_{n+1}}, and because $\beta_{\ell+1}=0$, that is the same as the smallest integer $\ell_2\ge n$ such that
$[e_{\ell+\ell_2-n},e_{n+1}]\neq 0$.
Consequently, $L_{(\ell+\ell_2+1)}\subseteq (L^2)^3$, and one sees that equality holds by a similar argument as the one we used for $(L^2)^2$,
considering Equation~\eqref{eq:ad_e_{n+j}} for $i>\ell$ rather than $i>n$.

This argument for $r=2$ can clearly be phrased into an inductive proof of the desired conclusion for all $r\ge 1$.
\end{proof}

The conclusion of Proposition~\ref{prop:constituent_lengths} extends a corresponding characterization of constituent lengths for algebras of type $1$,
which can be proved in a similar way.
However, that characterization holds unconditionally for algebras of type $1$.
In fact, in that case the hypothesis $[L^2,L_n]\subseteq L^{n+3}$ should be read as $[L^2,y]\subseteq L^4$,
which amounts to the chosen generator $y$ spanning the first two-step centralizer of $L$,
and that can always be achieved by a suitable choice of generators.
In contrast, the assumption $\ell>2n$, or $\beta_{n+1}=0$, is a genuine restriction for algebras of type $n>1$.
In particular, for algebras of type $2$ in characteristic zero, as considered in~\cite{ShZe:narrow-Witt},
one has $\beta_3\neq 0$ for the algebra $W$, as well as for one of the two metabelian algebras.
For algebras of type $2$ in odd characteristic, it was observed in~\cite{CVL00} that ``a theory of constituents is only possible when $\beta_3=0$,''
and Section~3 of~\cite{CVL00} was devoted to dealing with algebras of type $2$ where $\beta_3\neq 0$.
Fortunately, in the case $n=p$ of main interest in this paper (after the special case of $p=2$ in~\cite{CVL03}), the condition $\beta_{n+1}=0$ is not restrictive
because it can always be achieved by passing to a translated algebra, as we will see in Proposition~\ref{prop:translation}.

To conclude this section we discuss a competing definition of constituents for an algebra of type $p$,
equal to the characteristic, which was adopted in~\cite{Sca:thesis}.
There, the constituents in our sense were referred to as {\em fake constituents}.
{\em Scarbolo's constituents,} as we call them here referring to~\cite[Subsection~2.2]{Sca:thesis} for a formal definition,
also partition the sequence of $L$ into adjacent blocks, and their entries are all zero except possibly for the last $p$.
Differently from our constituents, Scarbolo's constituents are defined so they always end with a nonzero entry,
with the following entry being zero being the first entry of the next constituent.
It turns out that
each of Scarbolo's constituents ends where our constituent ends, or just one step earlier in case our constituent ends with a zero entry.
This is because the last entry of one of our constituents may vanish, but the last two entries cannot both vanish.
However, proving this fact is hard and requires a crucial part of the proof of Theorem~\ref{thm:classification}.

Scarbolo's constituents would carry a marginal advantage in a description of the algebras $L$ of the family $\mathcal{E}$
see Remark~\ref{rem:fake_constituents} for details.
Their use would affect the statements of Theorems~\ref{thm:first_length} and~\ref{thm:exceptional}, notably because
the length of Scarbolo's first constituent need not be even.
However, each of Scarbolo's constituents needs to start with a zero to be defined, a requirement which our definition avoids.
Consequently, Scarbolo's constituents would generally need to be strictly longer than $n$, whilst our definition
allows $n$ as the shortest possible length of a constituent (and $2n$ as a possible value for $\ell$, as discussed above).
In particular, our Proposition~\ref{lemma:Ugolini} would only make sense if every constituent length $m$ exceeds $n$,
despite being true without that restriction (with our proof, as we discussed in Remark~\ref{rem:Ugolini}).

The lengths of Scarbolo's constituents do admit an intrinsic characterization as relative codimensions in a certain sequence of ideals,
and we refer to~\cite[Subsection~2.2]{Sca:thesis} for details, but the recursive definition of those ideals is much more involved than
the Lie powers of $L^2$ used in our Proposition~\ref{prop:constituent_lengths}.
Finally, both definitions of constituents necessarily coexist in~\cite{Sca:thesis},
because most proofs make reference to our constituents rather than his.
Thus, in this paper we have opted to forgo the one marginal descriptive advantage given by Scarbolo's constituents
to the benefit of clarity and simplicity.

\section{Constituents of algebras of type $p$}\label{sec:type_p}

From now on in this paper we restrict attention to algebras of type $p$, equal to the characteristic.
Because $\binom{p-1}{i} \equiv (-1)^i \pmod{p}$ for $0\le i<p$,
the special constituents satisfying Equation~\eqref{eq:ordinary_from_type_1} are those for which the last $p$ entries are equal.
We call such constituents {\em ordinary}
and, if needed, {\em ending in $\lambda$,} where $\lambda$ is the common value of those entries.
The following result combines the special case $n=p$ of Proposition~\ref{lemma:Ugolini}, and the converse implication which we proved in Section~\ref{sec:constituents}.

\begin{prop}\label{OrdinaryEnding}
Let $L$ be a graded Lie algebra of maximal class of type $p$ over a field of odd characteristic $p$.
Then $L$ is a graded subalgebra of an uncovered algebra of type $1$ if and only if all its constituents are ordinary.
\end{prop}

When $n=p$ (and more generally when $n$ is a power of $p$, which we will not discuss here)
an important reduction is possible which is not available for arbitrary $n$.
This was discovered in~\cite[Lemma~2.3]{CVL03}, and allowed there a classification of algebras of type $2$ in characteristic $2$.
The following result is a straightforward extension of that to arbitrary prime characteristic $p$.

\begin{prop}\label{prop:translation}
Let $L$ be an algebra of type $p$ with sequence $(\beta_i)_{i>p}$.
Then for any $\delta \in F$ there exists an algebra of type $p$ with sequence $(\beta_i+\delta)_{i>p}$.
\end{prop}

\begin{proof}
Let $L$ be an algebra of type $p$ over a field $F$ of characteristic $p$.
Let $z$ and $e_p$ be its generators chosen as usual and let $(\beta_i)_{i > p}$
be the corresponding sequence.
Embed $L$ into an associative algebra $A$, for example its universal enveloping algebra, with respect to the Lie bracket $[x,y]=xy-yx$.
In this associative algebra we will choose a new algebra of type $p$, as a Lie subalgebra.
To avoid ambiguity with our notation for iterated Lie products we will denote by $x^{[p]}$ the associative $p$th power of an element $x$ of $A$,
and note that $[y,x^{[p]}]$ equals $[y,x^p]$ (where the latter stands for $[y,x,\ldots,x]$ like elsewhere in this paper).
For $\delta\in F$ consider the Lie subalgebra $L(\delta)$ of $A$ generated by $z$ and $e_p'=e_p+\delta z^{[p]}$.
Because $[z^{[p]},z]=0$, for $i>p$ we have
$e_i'=[e_p',z^{p-i}]=e_i$,
and
\[
[e_i',e_p']
=[e_i,e_p]+\delta[e_i, z^{[p]}]
=\beta_ie_{i+p}+\delta[e_i, z^p]=(\beta_i +\delta)e_{i+p}'.
\]
We conclude that $L(\delta)=L_1 \oplus F e_p' \oplus \bigoplus_{i>p}L_i$ is an algebra of type $p$,
with sequence $(\beta_i+\delta)_{i>p}$.
\end{proof}

We refer to $L(\delta)$ as a {\em translate} of $L$.
Proposition~\ref{prop:translation} is one reason why, as we mentioned in Section~\ref{sec:sequence}, one should resist the temptation of extending the definition
of the sequence $(\beta_i)_{i>p}$ of an algebra of type $p$ to include an element $\beta_p=0$ according to the defining rule
$[e_p,e_p]=\beta_p e_{2p}$, as $\beta_p$ would remain unaffected by translation.
The notation $x^{[p]}$ used for associative $p$th powers in the above proof hints at the fact that the construction of $L(\delta)$ could also take place
within any restricted Lie algebra containing $L$ as a subalgebra.
The proof of Proposition~\ref{prop:translation} would clearly extend almost verbatim to an analogous statement for algebras of type a power of $p$,
but we have no use for such generalization in this paper.

The translate $L(\delta)$, with $\delta\neq 0$, is isomorphic with $L$ precisely when
$\lambda\beta_i=\beta_i+\delta$ for some $\lambda\in F^*$ and for all $i>p$.
This means $\beta_i=\delta/(\lambda-1)$ for all $i>p$, hence $L$ has nonzero constant sequence.
As we discussed in Section~\ref{sec:sequence}, all such algebras are isomorphic, and have a maximal abelian ideal of codimension two.
However, if the sequence of $L$ is not constant, then $L$ and $L(\delta)$ are not isomorphic.
They may even have different constituent lengths, see the discussion below, and Remark~\ref{rem:m=p-1} for a concrete example.

Now let $N=\bigoplus_{i \ge 1} N_i$ be an uncovered algebra of type $1$,
and let $L$ hence be the subalgebra (of type $p$) generated by $z$ and $e_p$.
According to their definition in Section~\ref{sec:constituents}, the constituent lengths of $L$ equal those of $N$.
We study the first constituent length of a translate $L(\delta)$ of $L$,
which requires a few facts from the theory of algebras of type $1$ in~\cite{CMN,CN}.

\begin{prop}\label{prop:translated_subalgebra}
Let $N$ be an algebra of type $1$, with $(N^2)^2\neq 0$, and let $L$ be a graded subalgebra of type $p$.
Then $L$ has first constituent length $2q$, for some power $q$ of $p$.
Furthermore, a proper translate $L(\delta)$ of $L$, hence with $\delta\neq 0$,
has first constituent of length $2q$, or $q+p$ for some power $q$ of $p$.
\end{prop}

The values for the first constituent length $\ell$ obtained in Proposition~\ref{prop:translated_subalgebra}
account for part of the possibilities for $\ell$ allowed by Theorem~\ref{thm:first_length}.
The remaining values in Theorem~\ref{thm:first_length}, which are $q<\ell<q+p$, with $q$ a power of $p$ and $\ell$ even,
occur for the exceptional algebras of the family $\mathcal{E}$
(see Theorem~\ref{thm:exceptional} and Section~\ref{sec:length2q}).

\begin{proof}
It is known from~\cite{CMN} that, because $N$ is not metabelian, its first constituent has length $2q$ for some power $q$ of $p$,
and hence $L$ has also first constituent length $2q$.
In terms of the sequence $(\alpha_i)_{i>p}$ of $N$ this means that
its earliest nonzero element is $\alpha_{2q}$, and we may conveniently assume that to equal $1$.
According to Equation~\eqref{eq:ordinary_from_type_1} with $n=p$, the corresponding sequence $(\beta_i)_{i>p}$ for $L$ satisfies
$\beta_i=0$ for $p<i\le 2q-p$ and $\beta_i=1$ for $2q-p<i\le 2q$.
If $q>p$ then any translate $L(\delta)$ with $\delta\neq 0$ has $\beta_{p+1}+\delta\neq 0$,
hence its first constituent has length $2p$.

Thus, we may now assume that $N$ has first constituent length $2p$.
Then according to~\cite[Theorem~5.5]{CMN} any constituent of $N$ has length $2p$, $2p-1$, or $p$,
and the second constituent has length $p$ because of our blanket hypothesis $p>2$.
Any translate $L(\delta)$ with $\delta\neq -1$ will also have first constituent length $2p$,
hence consider $L(-1)$.
Its first constituent length $\ell$ is the smallest integer $\ell>p$ such that $\beta_{\ell-p+1}\neq 1$,
and we may assume such an integer exists, otherwise $L(-1)$ is metabelian with a maximal abelian of codimension $2$ (see Section~\ref{sec:sequence}).

If $\beta_{\ell-p+1}\neq 0$, then $\beta_{\ell-p+1}=\alpha_{\ell}$
marks the earliest occurrence of a two-step centralizer of $N$ different from the first and the second,
and hence $\ell$ equals twice a power of $p$ according to~\cite[Step~4.10]{CN}.
If $\beta_{\ell-p+1}=0$, then $\alpha_{\ell-p}$ is at the end of the last
in the longest sequence of consecutive constituent of $N$ of length $p$ starting with the second,
and no two-step centralizer other than the first and the second has occurred until then.
Thus, $\ell=(m+3)p$, where the sequence of constituent lengths of $N$, and hence of $L$, starts with $2p,p^m,2p$ or $2p,p^m,2p-1$.
According to~\cite[Lemma~5.1]{CN}, in the former case $m=2r-3$, where $r$ is a power of $p$.
Also, according to~\cite[Lemma~5.4]{CN}, in the latter case $m=r-2$, where again $r$ is a power of $p$.
\end{proof}

\begin{rem}\label{rem:L(-1)}
In the case of Proposition~\ref{prop:translated_subalgebra} where a proper translate $L(\delta)$
has first constituent $2q$, with $q>p$ (whence $\delta=-1$, as the proof shows),
according to our Theorem~\ref{thm:length2q}
there is an algebra $M$ of type $1$ such that $L(-1)$ is a graded subalgebra of $M$.
Note that $M$ is not isomorphic to $N$, as their fist constituents have different lengths $2q$ and $2p$.
However, using concepts from~\cite{CMN}, the algebras $N$ and $M$ have isomorphic deflations $N^\downarrow\cong M^\downarrow$.
Consequently, $M$ can be obtained from $N$ by deflation followed by an appropriate inflation.
The fact that a proper translate of a subalgebra of type $p$ of an algebra of type $1$ having first constituent length $2p$,
may be itself a subalgebra of an algebra of type $1$,
highlights some of the complications of the case $q=p$,
which is avoided in Theorem~\ref{thm:classification} thanks to its hypothesis $\ell>4p$.
\end{rem}

One important respect in which our proof of Theorem~\ref{thm:first_length} is simpler than the original proof in~\cite{Sca:thesis}
is that the main part of our argument focuses on the adjoint action of one particular element of the algebra,
as opposed to using a range of relations (which were often motivated by computer calculations).
The resulting greater clarity entails a more compelling explanation for the allowed values of $\ell$ in the conclusion.
To justify this statement and as an introduction to the key argument to be used in the main part of the proof of Theorem~\ref{thm:first_length},
in Subsection~\ref{subsec:main}, we sketch here a proof that the
first constituent of an algebra of type $1$ must have length $2q$ for some power $q$ of the characteristic.
This is a more efficient variation of the original proof given in~\cite{CMN}.
We direct the interested reader to~\cite{Mat:chain_lengths} for further arguments of this flavour.

Thus, consider an algebra $L$ of type $1$, which we may assume uncovered, whence
$L_1$ has a basis given by $e_1$ and $z$ as in Section~\ref{sec:sequence},
and each homogeneous component $L_i$ with $i>1$ is spanned by $e_i=[e_1,z^{i-1}]$,
and $e_1$ is chosen such that $[e_2,e_1]=0$.
Let $\ell$ be the length of the first constituent, which we know to be even from Lemma~\ref{lemma:ell_even}, and means
$[e_i,e_1]=0$ for $0<i<\ell$ but
$[e_{\ell},e_1]\neq 0$.
After replacing $e_1$ with a suitable scalar multiple we may assume $[e_{\ell},e_1]=e_{\ell+1}=[e_{\ell},z]$.
Let also $\ell_2$ be the length of the second constituent, which means
$[e_j,e_1]=0$ for $0<j<\ell_2$ but
$[e_{\ell+\ell_2},e_1]\neq 0$.
It will be convenient to consider $x=z-e_1$, because $[e_1,x^{\ell}]=[e_{\ell},x]=0$.

Then for $0<j<\ell_2$ we have
\begin{align*}
0=[e_j,[e_1,x^{\ell}]]
=
\sum_{i=0}^{\ell}(-1)^i\binom{\ell}{i}
[e_j,x^i,e_1,x^{\ell-i}]
=
(-1)^{\ell-j}\binom{\ell}{\ell-j}
[e_{\ell},e_1,x^j],
\end{align*}
because all Lie products in the sum vanish except for the term with $i=\ell-j$.
Because $[e_{\ell},e_1,x^j]=e_{\ell+j+1}\neq 0$
we infer $\binom{\ell}{j}\equiv 0\pmod{p}$ for $0<j<\ell_2$.
However, an analogue of Lemma~\ref{lemma:constituent_bound} for algebras of type $1$, which can be proved in the same way (see~\cite[Lemma~5]{Mat:chain_lengths}),
would tell us that $\ell_2\ge\ell/2$.
Consequently, we have
$\binom{\ell}{j}\equiv 0\pmod{p}$
for $0<j<\ell/2$.
Now  Lucas' Theorem (which we recall at the beginning of Section~\ref{sec:polynomials}) together with $\ell$ being even (and $p$ odd)
easily implies that $\ell$ equals twice a power of $p$.

One crucial fact to take home from this argument, in view of replicating it for algebras of type $n$ in Section~\ref{sec:first_length_proof},
is that it exploits the adjoint action of $[e_1,x^{\ell}]$, a Lie product of formal degree $\ell+1$, which happens to vanish here.
At the expense of one additional calculation the argument could be rephrased in terms of the generators $e_1$ and $z$ for $L$
(that is, with $e_1$ possibly not centralizing $L_2$, and $z$ not centralizing any $L_i$),
where we would exploit the adjoint action of the element $[e_1,z^{\ell}]=e_{\ell+1}$.
Similarly, the main argument in our proof of Theorem~\ref{thm:first_length}, which we give in Subsection~\ref{subsec:main},
will exploit the action of the element $e_{\ell+1}=[e_p,z^{\ell-p+1}]$ in an algebra of type $p$.

Because constituents in an algebra of type $p$ are considerably more complicated than for algebras of type $1$
(having up to $p$ nonzero trailing entries),
it will be convenient to express the conclusion of that argument in terms of a condition on a certain polynomial,
whose consequences will be studied in Section~\ref{sec:polynomials}.
A somehow similar polynomial framing of the conclusion of the above argument for algebras of type $1$,
which is given in~\cite{Mat:chain_lengths}, would be
that the Lie bracket calculation entails
$(x-1)^{\ell}=x^{\ell}+ax^{\ell/2}+1$ in the polynomial ring $F[x]$, for some $a\in F$.

\section{Polynomial arguments}\label{sec:polynomials}

The main argument in our proof of Theorem~\ref{thm:first_length},
in Subsection~\ref{subsec:main},
will prove the vanishing of a certain range of coefficients in the product
$(x-1)^{\ell-p+1}g(x)$,
where $g(x)$ is a polynomial of degree $p-1$ encoding the last $p$ entries of the first constituent of an algebra of type $p$.
In this section we present a result on polynomials, Theorem~\ref{thm:polynomials},
which under such condition on $(x-1)^{\ell-p+1}g(x)$ restricts the possibilities for $\ell$ to those in the conclusion of Theorem~\ref{thm:first_length}
and a few more, the latter to be excluded later by different calculations in $L$.
In this section only we relax our blanket hypothesis on $p$ being odd.

We will make frequent use of Lucas' theorem, a basic tool for evaluating a binomial coefficient $\binom{a}{b}$ modulo a prime $p$:
if $a$ and $b$ are non-negative integers with $p$-adic expansions $a=a_0+a_1p+\cdots+a_rp^r$ and $b=b_0+b_1p+\cdots+b_rp^r$ (where $0\le a_i,b_i<p$), then
\[
\binom{a}{b}=\prod_{i=0}^{r}\binom{a_i}{b_i} \pmod{p}.
\]
A simple consequence of Lucas' Theorem, which is alone sufficient for many purposes, is that $\binom{a}{b}$ vanishes modulo $p$ if,
and only if, $b_i>a_i$ for some $i$.

Following standard notation we denote by $[x^j] g(x)$ the coefficient of $x^j$ in a polynomial $g(x)$.
Let $F$ be a field of positive characteristic $p$.
All polynomials considered in this section will be viewed as having coefficients in $F$.
For example, if $k$ is a positive integer such that $[x^j] (x-1)^k=0$ for $k/2\le j<k$, then it is easy to deduce that $k$ must be a power of $p$.
In fact, the condition then holds on the range $0<j<k$ because of the binomial identity $\binom{k}{k-j}=\binom{k}{j}$,
and then the conclusion follows from Lucas' Theorem (or alternate arguments).
If the condition is only assumed for the range $k/2<j<k$ then $k$ must be either a power of $p$ or twice a power of $p$
(see~\cite[Lemma~1]{Mat:chain_lengths} for two different proofs).
The following preparatory result gives precise restrictions on the exponent $k$ under the condition that about the upper half of coefficients
of the polynomial $(x-1)^k(x-a) \in F[x]$ vanish,
with the necessary exception of the leading coefficient.
The precise formulation of this condition comes in two variants, differing only when $k$ is odd, both of which will be needed in the proof
of Theorem~\ref{thm:polynomials}.

\begin{lemma}\label{lemma:polynomials}
Let $F$ be a field of characteristic $p$ and let $x-a \in F[x]$. Suppose that, for a natural number $k>1$, we have
\begin{equation}\label{eq:range_lemma}
[x^j](x-1)^k(x-a)=0 \quad\text{for } k/2+1\le j\le k.
\end{equation}
Then the pair $(k,a)$ is one of the following: either $(2,-2)$, or $(3,-3)$, or $(q,0)$, or $(q-1,1)$, or $(2q-1,1)$, where $q$ is a power of $p$.

Moreover, if Equation~\eqref{eq:range_lemma} holds in the
range $(k+1)/2\le j\le k$, then the pairs $(3,-3)$ and $(2q-1,1)$ can be excluded from the conclusion.
\end{lemma}

The cases $(2,-2)$ and $(3,-3)$ are genuine possibilities because of the polynomials $(x-1)^2(x+2)=x^3-3x+3$ and $(x-1)^3(x+3)=x^4-6x^2+8x-3$.

\begin{proof}
Assume $k>3$, the remaining cases being easy to deal with directly.
The coefficients of $x^k$ and $x^{k-1}$ in the product $(x-1)^k(x-a)$ are $-a-k$ and $ak+\binom{k}{2}$, respectively.
The vanishing of the former yields $a=-k$ (in the field $F$), and when this is substituted in the latter the vanishing of that
implies that $k$ is congruent to either $0$ or $-1$ modulo $p$.
If $k\equiv 0\pmod{p}$, whence $a=0$, Equation~\eqref{eq:range_lemma} is equivalent to
$[x^j](x-1)^k=0$ for $k/2\le j<k$, and so $k$ must be a power of $p$.
If $k\equiv -1\pmod{p}$, whence $a=1$, Equation~\eqref{eq:range_lemma} reads
$[x^j](x-1)^{k+1}=0$ for $k/2+1\le j\le k$,
and so $k+1$ must be either a power of $p$ or twice a power of $p$.
The latter is excluded when the condition's range is extended to
$(k+1)/2\le j\le k$.
\end{proof}

The next result, which is the main goal of this section, is similar to Lemma~\ref{lemma:polynomials} but replaces the factor $x-a$
in its hypothesis with an arbitrary polynomial of degree $p-1$.
The condition on the coefficients is modelled precisely on what will be needed in our proof of Theorem~\ref{thm:first_length} in the next section.

\begin{thm}\label{thm:polynomials}
Let $F$ be a field of positive characteristic $p$ and let $k>p+1$ be a natural number.
Let $g(x)\in F[x]$ be a monic polynomial of degree $p-1$ such that
\begin{equation}\label{eq:range}
[x^j](x-1)^kg(x)=0 \quad\text{for } (k+p)/2\le j<k.
\end{equation}
Then either $p+1<k<2p$, or $2p<k<3p$, or $k=3p+1$, or $k=2q-p+1$, or $q-p<k<q+p$, where $q>p$ is a power of $p$.

Moreover, if $k=2q-p+1$ or $k=q-p+1$, then $g(x)=(x-1)^{p-1}$.
If $k=q-p+k_0$, with $0<k_0<p$, then $(x-1)^{p-k_0}$ divides $g(x)$.
If $k=q+k_0$ with $0<k_0<p$, then $x^{k_0}$ divides $g(x)$.
\end{thm}

Our hypothesis $k>p+1$ in Theorem~\ref{thm:polynomials} is due to the condition expressed in Equation~\eqref{eq:range} being void otherwise.

\begin{proof}
Write $k=k'p+k_0$ with $0\le k_0<p$, and write $g(x)=\sum_{i=0}^{p-1}g_ix^i$, whence $g_{p-1}=1$.
For a polynomial $f(x)\in F[x]$, and an integer $i$,
denote by $S_i\bigl(f(x)\bigr)$ the polynomial obtained from $f(x)$ by discarding all terms
where $x$ appears with an exponent not congruent to $i$ modulo $p$.
In particular, we have $f(x)=\sum_{i=0}^{p-1}S_i\bigl(f(x)\bigr)$.

We first deal with the simplest case, which is when $k_0=0$, that is, $k$ is a multiple of $p$.
Then we have
\[
(x-1)^{k}g(x)=(x^p-1)^{k'}g(x),
\]
and hence $S_i\bigl( (x-1)^{k}g(x) \bigr)=(x^p-1)^{k'}g_ix^i$ for each $i$.
Taking $i=p-1$ our hypothesis yields
\begin{equation*}
[x^j](x-1)^{k'}=0 \quad\text{for } k'/2\le j<k',
\end{equation*}
where the inequality on the left is due to $p-1\ge p/2$.
As discussed earlier this forces $k'$ to be a power of $p$,
and hence $k$ as well, say $k=q>p$.
Clearly we cannot get any information on $g(x)$ in this case.

Now suppose $k_0=1$. Then
\[
(x-1)^{k}g(x)=(x^p-1)^{k'}(x-1)g(x),
\]
and hence
$S_i\bigl( (x-1)^{k}g(x) \bigr)=(x^p-1)^{k'}(g_{i-1}-g_i)x^i$ for $0<i<p$,
and
$S_0\bigl( (x-1)^{k}g(x) \bigr)=(x^p-1)^{k'}(x^p-g_0)$.
For $i=0$ Equation~\eqref{eq:range} implies
\begin{equation*}
[x^j](x-1)^{k'}(x-g_0)=0 \quad\text{for }k'/2+1\le j\le k',
\end{equation*}
because precisely these values of $j$ satisfy $(k'p+1+p)/2\le jp<k'p+1$.
Because of our assumption $k>p+1$ we have $k'>1$.
Lemma~\ref{lemma:polynomials} shows that
$(k',g_0)$ equals either $(2,-2)$, or $(3,-3)$, or $(q',0)$, or $(q'-1,1)$, or $(2q'-1,1)$, for some power $q'\ge p$ of $p$.
Consequently, $k$ can only assume one of the following values: $k=2p+1$, $k=3p+1$, $k=q+1$, $k=q-p+1$, or $k=2q-p+1$, where $q=q'p>p$.

We obtain additional information on $g(x)$ when $k_0=1$, that is, when $k=q-p+1$ or $k=2q-p+1$.
In fact, for each $0<i<p$, Equation~\eqref{eq:range} implies
\begin{equation*}
[x^j](g_{i-1}-g_i)(x^p-1)^{k'}=0 \quad\text{for }(k'+1)/2\le j<k',
\end{equation*}
because these values of $j$ satisfy $(k'p+1+p)/2\le jp+i<k'p+1$.
The lower end of the range can be extended for specific values of $i$,
but we need not be precise here because it will suffice to look at $j=k'-1$.
In fact, for this value of $j$ we find $g_i=g_{i-1}$, hence starting from $g_0=1$ we inductively find $g_i=1$ for all $i$.
In conclusion, we find $g(x)=(x-1)^{p-1}$.

More generally, suppose now $k=k'p+k_0$, where $1<k_0<p$. Then
\begin{equation*}
(x-1)^k g(x)= (x^p-1)^{k'} (x-1)^{k_0} g(x),
\end{equation*}
whence
\begin{align*}
S_{k_0-1}((x-1)^k g(x)) &=(x^p-1)^{k'} \biggl(x^p+ \sum_{s=0}^{k_0-1} (-1)^{k_0-s} \binom{k_0}{s}g_{k_0-1-s} \biggr)x^{k_0-1}.
\end{align*}
Equation~\eqref{eq:range} implies
\begin{equation*}
[x^j] (x-1)^{k'} (x-a)=0
\quad\text{for } (k'+1)/2 \le j \le k',
\end{equation*}
where
$a=\sum_{s=0}^{k_0-1}(-1)^{k_0+1-s} \binom{k_0}{s}g_{k_0-1-s}$,
because these values of $j$ satisfy $(k'p+k_0+p)/2 \le jp+k_0-1 < k'p+k_0$.
Lemma~\ref{lemma:polynomials} then tells us that $k'$ can only equal $1$, $2$, $q'$ or $q'-1$, where $q'>1$ is a power of $p$.
(It also tells us corresponding values for $a$, which we will not need here.)
Consequently, either $k=p+k_0$ or $k=2p+k_0$ or $k=q+k_0$ or $k=q-p+k_0$, where $q=q'p>p$.

It remains to prove the additional information on $g(x)$ stated in Theorem~\ref{thm:polynomials} for the cases where $q-p<k<q+p$.
Clearly, no information can be extracted from Equation~\eqref{eq:range} when $k=q$.
The higher half of the range is much easier to deal with.
In fact, if $k=q+k_0$ with $0<k_0<p$,
because $(x-1)^k=(x^{q}-1)(x-1)^{k_0}$
Equation~\eqref{eq:range} implies
\[
[x^j](x-1)^{k_0} g(x)=0 \quad\text{for } 0 \le j< k_0.
\]
Because $x^{k_0}$ and $(x-1)^{k_0}$ are coprime it follows that $x^{k_0}$ divides $g(x)$, as desired.

Now assume $k=q-p+k_0$ with $0<k_0<p$.
(The case $k_0=1$ was proved earlier, but the following argument includes it.)
Here Equation~\eqref{eq:range} reads
\begin{equation*}
[x^j] (x-1)^{q-p}(x-1)^{k_0}g(x)=0 \quad\text{for } (q+k_0)/2\le j<q-p+k_0.
\end{equation*}
Expanding $(x-1)^{q-p}$ as $x^{q-p}+x^{q-2p}+\cdots+x^p+1$,
and restricting the range of $j$ to an interval where only the first two terms of the expansion matter, as long as $q-2p+k_0\ge (q+k_0)/2$ we infer
\begin{equation*}
[x^j] x^{q-2p}(x^p+1)(x-1)^{k_0}g(x)=0 \quad\text{for } q-2p+k_0\le j < q-p+k_0,
\end{equation*}
which amounts to
\[
[x^j] (x^p+1)(x-1)^{k_0}g(x)=0 \quad\text{for } k_0\le j < p+k_0.
\]
The assumption $q-2p+k_0\ge (q+k_0)/2$ holds except when the pair $(p,k)$ equals $(2,3)$, $(3,7)$, or $(3,8)$.
The first of those three possibilities is excluded by our hypothesis $k>p+1$, and
the desired conclusion is easy to obtain by inspection in the other two cases.
If we further restrict the range of $j$ in the above equation, looking only at terms of degree less than $p$, we find
\begin{equation*}
[x^j] (x-1)^{k_0}g(x)=0 \quad\text{for } k_0\le j < p.
\end{equation*}
Hence $(x-1)^{k_0}g(x)=x^pa(x)+b(x)$, with $a(x)$ and $b(x)$ polynomials of degree less than $k_0$.
Because $1$ is a root of $x^pa(x)+b(x)$ with multiplicity at least $k_0$, the derivatives
of this polynomial up to order $k_0-1$ also have $1$ as a root.
Since $x^p$ has zero derivative the analogous conclusion holds for the corresponding derivatives of $a(x)+b(x)$.
Because $k_0<p$ this implies that $a(x)+b(x)$ has $1$ as a root with multiplicity at least $k_0$, which exceeds its degree,
whence $a(x)+b(x)$ is the zero polynomial.
Consequently,
$(x-1)^{k_0}g(x)=x^pa(x)-a(x)=(x-1)^p a(x)$,
and the desired conclusion follows.
\end{proof}

\section{Proof of Theorem~\ref{thm:first_length}}\label{sec:first_length_proof}

Let $L$ be an algebra of type $p$, hence generated by $z$ and $e_p$ with the usual meaning, and with first constituent length $\ell>4p$.
Our main argument, in Subsection~\ref{subsec:main}, will use information on the first and second constituent,
together with Lemma~\ref{lemma:constituent_bound},
to show that the last $p$ entries of the first constituent satisfy a linear recurrence (in their range).
This can be interpreted as a condition on the coefficients of a certain polynomial which will allow an application of Theorem~\ref{thm:polynomials}.
However, the conclusions of this argument will allow for some values for $\ell$, namely $q+p<\ell<q+2p-1$,
which are not genuine possibilities, and require further calculations to exclude.
We do that in the second part of the proof, which forms Subsection~\ref{subsec:spurious} and is considerably more complicated.
Subsection~\ref{subsec:further_info} goes beyond the proof of Theorem~\ref{thm:first_length}
to extract some additional information from Theorem~\ref{thm:polynomials} which will be of later use.

\subsection{The main argument}\label{subsec:main}
For convenience we state Equations~\eqref{eq:first_const}--\eqref{eq:second_const_end},
which encode the first and second constituent of $L$ (including the normalization $\beta_{\ell-n+1}=1$),
with $p$ in place of $n$:
\begin{subequations}
\begin{align}
[e_i, e_p]&=0\quad\text{for }p<i\le\ell-p,
 \label{eq:first_const}\\
[e_{\ell-p+1},e_p]&=e_{\ell+1},
 \label{eq:first_const_end}\\
[e_{\ell+j},e_p]&=0\quad\text{ for }0<j\le \ell_2-p,
 \label{eq:second_const}\\
[e_{\ell+\ell_2-p+1},e_p]&=\beta_{\ell+\ell_2-p+1}e_{\ell+\ell_2+1}\ne 0.
 \label{eq:second_const_end}
\end{align}
\end{subequations}
Note that at this stage we have no information on
$[e_i, e_p]$ for $\ell-p+1<i\le\ell$
(apart from it being a scalar multiple of $e_{p+i}$)
as we do not know the corresponding entries $\beta_{\ell-p+2},\ldots,\beta_{\ell}$ of the first constituent.
We will use Equations~\eqref{eq:first_const}--\eqref{eq:second_const_end}
to obtain information on the adjoint action of $e_{\ell+1}$.
(This loosely plays the role of the relation $[e_1,x^\ell]=0$ used for the case of algebras of type $1$,
as we sketched at the end of Section~\ref{sec:type_p}.)

Our first observation is that $e_{\ell+1}$ centralizes $e_{j}$ for $p\le j<\ell_2$, namely,
\[
[e_{\ell+1},e_{j}]
=[e_{\ell+1},[e_p,z^{j-p}]]
=\sum_{i=0}^{j-p} (-1)^i\ \binom{j-p}{i} [e_{\ell+1+i}, e_p, z^{j-p-i}]
=0,
\]
because $p+i\le j<\ell_2$.
Consequently, after interchanging the entries of $[e_{\ell+1},e_{j}]$ and expanding by the generalized Jacobi identity we find
\begin{align*}
0=[e_{j},[e_p,z^{\ell-p+1}]]
=\sum_{i=0}^{\ell-p+1} (-1)^i \binom{\ell-p+1}{i} [e_{j+i}, e_p, z^{\ell-p+1-i}].
\end{align*}
Now note that $j\le j+i\le j+\ell-p+1\le\ell+\ell_2-p$.
Hence, all Lie products involved in the above summation vanish except possibly for those with
$\ell-p<j+i\le\ell$,
that is,
$\ell-p-j<i\le\ell-j$.
We conclude
\begin{equation*}
0= \sum_{i=0}^{p-1} (-1)^{j+i} \binom{\ell-p+1}{\ell-j-i} [e_{\ell-i}, e_p, z^{j-p+1+i}],
\end{equation*}
whence
\begin{equation*}\label{linear recurrence}
\sum_{i=0}^{p-1} (-1)^{j+i} \binom{\ell-p+1}{\ell-j-i} \beta_{\ell-i}=0
\qquad \text{for } p\le j<\ell_2.
\end{equation*}
The left-hand side equals the coefficient of $x^{\ell-j}$ in the product $(1-x)^{\ell-p+1}g(x)$, where
$g(x)=\beta_{\ell-p+1}x^{p-1}+\beta_{\ell-p+2}x^{p-2}+\cdots+\beta_{\ell}$,
a polynomial of degree $p-1$, which is actually monic due to our normalization $\beta_{\ell-p+1}=1$.
Thus, we have found that $g(x)$ satisfies
\begin{equation}\label{fromLinRecurToZeroCoef}
[x^j] (x-1)^{\ell-p+1} g(x)=0 \qquad \text{for } \ell-\ell_2<j\le\ell-p.
\end{equation}
Because $\ell_2\ge \ell/2$ according to Lemma~\ref{lemma:constituent_bound}, and setting $k=\ell-p+1$, this implies
\begin{equation*}
[x^j] (x-1)^k g(x)=0 \text{\quad for } (k+p+1)/2\le j<k.
\end{equation*}
Because $k$ is even according to Lemma~\ref{lemma:ell_even}, the range is equivalent to $(k+p)/2\le j<k$,
and so this condition is equivalent to Equation~\eqref{eq:range} in Theorem~\ref{thm:polynomials}.
In addition, here $k>3p+1$ because of our hypothesis $\ell>4p$.
Thus, Theorem~\ref{thm:polynomials} yields that
either $k=2q-p+1$ or $q-p<k<q+p$, where $q>p$ is a power of $p$.
Because $k$ is odd the latter range is equivalent to $q-p+2\le k\le q+p-2$.
Thus, in terms of $\ell=k+p-1$ we conclude that either $\ell=2q$ or $q+1\le\ell\le q+2p-3$.

\subsection{Excluding spurious possibilities for $\ell$}\label{subsec:spurious}
In order to complete a proof of Theorem~\ref{thm:first_length} it remains to rule out the possibilities $q+p<\ell\le q+2p-3$
for the length of the first constituent,
which means $q+3\le k\le q+p-2$ in terms of $k$.

Thus, assume $k=q+k_0$, where $q>p$ is a power of $p$ and $k_0$ is an odd integer with $1<k_0<p$.
This easily implies $\ell_2\le q$ using Equation~\eqref{fromLinRecurToZeroCoef}, because $(x-1)^k=(x^q-1)(x-1)^{k_0}$ and hence
\[
[x^{k_0+p-1}](x-1)^k g(x)=-[x^{k_0+p-1}](x-1)^{k_0}g(x)=-1\neq 0.
\]
With some harder work we will show that this implies $\ell_2=q$ or $q-1$,
and then show how they lead to a contradiction.

Theorem~\ref{thm:polynomials} provides further information on $g(x)$, namely, that it is a multiple of $x^{k_0}$.
That means
\begin{equation*}
\beta_{q+p}=\cdots=\beta_{q+k_0+p-1}=0,
\end{equation*}
or, in words, that the last $k_0$ entries of the first constituent vanish.
Put differently, under our assumption on $\ell$, Equation~\eqref{eq:second_const} holds over an extended range, namely,
\begin{equation}\label{eq:second_const_plus}
[e_{\ell+j},e_p]=0\quad\text{ for }-k_0<j\le \ell_2-p.
\end{equation}
Let $s<p$ be largest such that $\beta_{q+s}\neq 0$.
By assumption we have $s\ge k_0$.
Recall that the first non-zero entry of the second constituent is $\beta_{q+k_0+\ell_2}$.

Because of our hypothesis $\ell>4p$, and because we know $\ell_2\ge\ell/2$ from Lemma~\ref{lemma:constituent_bound},
we have $\ell_2+k_0\ge s+p$.
Hence we may write
$e_{\ell_2+k_0-s}=[e_{\ell_2+k_0-s-p},e_p]$, and compute
\begin{align*}
[e_{\ell_2+k_0-s},e_{q+s}]
&=-[e_{q+s},e_{\ell_2+k_0-s}]
\\&=
-\sum_{j=0}^{\ell_2+k_0-s-p} (-1)^j \binom{\ell_2+k_0-s-p}{j}[e_{q+s+j},e_p,z^{\ell_2+k_0-s-p-j}]
\\&=
[e_{q+s},e_p,z^{\ell_2+k_0-s-p}]
=-\beta_{q+s} e_{q+\ell_2+k_0},
\end{align*}
where all remaining terms in the summation vanish because of our choice of $s$.
Consequently,
\begin{align*}
\beta_{q+s}[e_{\ell_2+k_0-s},e_{q+s+p}]
&=
[e_{\ell_2+k_0-s},[e_{q+s},e_p]]
\\&=
[e_{\ell_2+k_0-s},e_{q+s},e_p]-[e_{\ell_2+k_0-s},e_p,e_{q+s}]
\\&=
[e_{\ell_2+k_0-s},e_{q+s},e_p]
\\&=
-\beta_{q+s}[e_{q+\ell_2+k_0},e_p],
\end{align*}
where the second term in the second line vanishes because $\ell_2+k_0-s\le q<q+k_0$,
due to $\ell_2\le q$ as we proved earlier.
Because $\beta_{q+s}\neq 0$ we deduce
\begin{equation}\label{eq:s_even}
[e_{\ell_2+k_0-s},e_{q+s+p}]
=-[e_{q+\ell_2+k_0},e_p]\neq 0.
\end{equation}
However, we also have
\begin{align*}
[e_{\ell_2+k_0-s},e_{q+s+p}]
&=
\sum_{j=0}^{q+s} (-1)^j \binom{q+s}{j}[e_{\ell_2+k_0-s+j},e_p,z^{q+s-j}]
\end{align*}
The binomial coefficients in the above summation vanish modulo $p$ except for the terms with $0\le j\le s$ or $q\le j\le q+s$.
Because of the length of the second constituent all terms in the latter range vanish except when $j=q+s$.
Because $\ell_2\le q$, in the former range we have $\ell_2+k_0-s+j<q+k_0$, whence the corresponding Lie product vanishes,
except possibly for $j=s$ in case $\ell_2=q$.
Thus,
\[
[e_{\ell_2+k_0-s},e_{q+s+p}]
=
(-1)^s[e_{\ell_2+k_0},e_p,z^{q}]
+(-1)^{q+s}[e_{q+\ell_2+k_0},e_p].
\]
Combined with Equation~\eqref{eq:s_even}, if $s$ is odd this yields
\[
[e_{\ell_2+k_0},e_p,z^{q}]
=
2[e_{q+\ell_2+k_0},e_p]\neq 0,
\]
whence $\ell_2=q$ as observed above.

To deal with the case where $s$ is even we need a slightly different calculation, in degree one higher.
Similar calculations as in the previous case yield
\[
[e_{\ell_2+k_0-s+1},e_{q+s}]
=-[e_{q+s},e_{\ell_2+k_0-s+1}]
=-\beta_{q+s} e_{q+\ell_2+k_0+1},
\]
and then
\[
\beta_{q+s}[e_{\ell_2+k_0-s+1},e_{q+s+p}]
=
[e_{\ell_2+k_0-s+1},e_{q+s},e_p]
=
-\beta_{q+s}[e_{q+\ell_2+k_0+1},e_p],
\]
whence
\begin{equation}\label{eq:s_odd}
[e_{\ell_2+k_0-s+1},e_{q+s+p}]
=-[e_{q+\ell_2+k_0+1},e_p]\neq 0.
\end{equation}
We also have
\begin{align*}
[e_{\ell_2+k_0-s+1},e_{q+s+p}]
&=
\sum_{j=0}^{q+s} (-1)^j \binom{q+s}{j}[e_{\ell_2+k_0-s+1+j},e_p,z^{q+s-j}]
\end{align*}
As in the previous case the binomial coefficients in the above summation vanish modulo $p$ except for the terms with $0\le j\le s$ or $q\le j\le q+s$.
However, here Lie products in the latter range vanish except when $j=q+s-1$ and possibly $q+s$.
Also, Lie products in the former range vanish as long as $\ell_2+k_0-s+1+j<q+k_0$.
This holds with the possible exceptions of $j=s-1$ or $s$, which may only arise when $\ell_2=q$ (in both cases) or $q-1$ (in the latter case).
Thus,
\begin{align*}
[e_{\ell_2+k_0-s+1},e_{q+s+p}]
&=
-(-1)^{s}s[e_{\ell_2+k_0},e_p,z^{q+1}]
+(-1)^s[e_{\ell_2+k_0+1},e_p,z^{q}]
\\&\quad
+(-1)^{s}s[e_{q+\ell_2+k_0},e_p,z]
-(-1)^{s}[e_{q+\ell_2+k_0+1},e_p].
\end{align*}
Combined with Equation~\eqref{eq:s_odd}, if $s$ is even this yields
\[
0\neq s[e_{q+\ell_2+k_0},e_p,z]
=
s[e_{\ell_2+k_0},e_p,z^{q+1}]
-[e_{\ell_2+k_0+1},e_p,z^{q}],
\]
whence at least one of the terms at the right-hand side is nonzero,
implying $\ell_2=q$ or $\ell_2=q-1$, as claimed.

We will now use these conclusions to reach a contradiction.
Because of Equation~\eqref{eq:first_const} we have
\[
[e_{q+k_0-1},e_{q+1},e_p]=[e_{q+k_0-1}, e_p, e_{q+1}]-[e_ {q+1},e_p,e_{q+k_0-1}]=0,
\]
and hence
\begin{align*}
0=[e_{q+k_0-1},e_{q+1},e_p]=\sum_{j=0}^{q+1-p}(-1)^j \binom{q+1-p}{j}[e_{q+k_0-1+j},e_p,z^{q+1-p-j},e_p].
\end{align*}
The Lie products involved in this summation vanish for $j=0$ because of Equation~\eqref{eq:first_const},
and for $j>p-k_0$ because of Equation~\eqref{eq:second_const_plus} together with $\ell_2\ge q-1>q-p$.
However, the binomial coefficients in the remaining range vanish modulo $p$ except when $j=1$, and we conclude
\[
0=[e_{q+k_0},e_p,z^{q-p},e_p ]=[e_{2q+k_0},e_p].
\]
This shows $\ell_2\neq q$, and hence $\ell_2=q-1$.

To proceed further, note that $[e_{q+k_0+1},e_{q-1},e_p]=0$, because $[e_{q+k_0+1},e_{q-1}]$ is a multiple of $e_{2q+k_0}$.
Together with Equation~\eqref{eq:first_const} this yields
\[
[e_{q+k_0+1},e_p,e_{q-1}]=[e_{q+k_0+1},e_{q-1},e_p]+[e_{q-1},e_p,e_{q+k_0+1}]=0.
\]
Expanding the left-hand side as usual we find
\[
0
=\beta_{q+k_0+1}\sum_{j=0}^{q-1-p} (-1)^j \binom{q-1-p}{j} [e_{q+k_0+1+p+j},e_p,z^{q-1-p-j}],
\]
and
$\beta_{q+k_0+1}=(k_0+1)/2\cdot\beta_{q+k_0}\neq 0$
according to Equation~\eqref{eq:beta}.
The Lie products involved in this summation vanish
for $j<q-2-p$ because of Equation~\eqref{eq:second_const_plus} together with $\ell_2=q-1$.
The two remaining terms yield
\[
0=[e_{2q+k_0-1},e_p,z]+[e_{2q+k_0},e_p].
\]
We have already shown that the latter vanishes, hence the former vanishes as well, contradicting $\ell_2=q-1$.
This final contradiction shows that $q+p<\ell<q+2p-1$ as assumed in this subsection cannot occur,
and thus completes the proof of Theorem~\ref{thm:first_length}.

\subsection{Further information on the first constituent}\label{subsec:further_info}
According to the argument in the main part of the above proof, in Subsection~\ref{subsec:main},
Theorem~\ref{thm:polynomials} provides further information on $g(x)$ in the various cases for $\ell$.
In particular, when $\ell=2q$ we have $g(x)=(x-1)^{p-1}=\sum_{i=0}^{p-1}x^i$.
This says that in this case $\beta_{2q-p+1}= \cdots=\beta_{2q}=1$, hence the first constituent is ordinary.
This fact will be essential for our proof of Theorem~\ref{thm:length2q}, in the next section.
Moreover, if $\ell=q+p$ then the first constituent ends with a zero (that is, $\beta_{\ell}=0$).
Finally, if $q<\ell<q+p$ then $g(x)$ is a multiple of $(x-1)^{p-k_0}=(x-1)^{q+p-\ell-1}$.
This gives us partial information on the first constituent of the algebras in Scarbolo's exceptional family $\mathcal{E}$,
which we describe in Section~\ref{sec:exceptional}.

\section{Algebras with first constituent of length $2q$}\label{sec:length2q}

In this section we study Lie algebras of type $p$ for which $\ell=2q$
and we prove that they are subalgebras of Lie algebras of type $1$, thus proving Theorem \ref{thm:length2q}.
In order to do so, it suffices to prove that every constituent is ordinary,
because the desired conclusion will then follow from Proposition~\ref{OrdinaryEnding}.
We know from Subsection~\ref{subsec:further_info} that the first constituent is ordinary,
and we proceed inductively.
Thus, assume that the $(r-1)$th constituent is ordinary,
and recall from Lemma~\ref{lemma:constituent_bound} that no constituent length is less than $q$.

Set $s=\ell+\ell_2+\cdots+\ell_{r-1}$.
Summarizing our hypotheses, for the first constituent we have
$\beta_{2q-p+1}= \cdots=\beta_{2q}=1$,
which means
$[e_{2q-p+j},e_p]=e_{2q+j}$ for $0<j\le p$,
and for the $(r-1)$th constituent
$\beta_{s-p+1}=\cdots=\beta_{s}=\lambda$, for some nonzero $\lambda$.
We will prove a corresponding claim for the $r$th constituent. Write
\[
\beta_{s+\ell_r-p+1}=\lambda_1, \ \beta_{s+\ell_r-p+2}=\lambda_2, \ldots,\ \beta_{s+\ell_r}=\lambda_p
\]
for convenience.
Our goal is showing that these $p$ scalars are all equal.
We will need to distinguish two cases according as whether $\ell_r$ equals $q$ or is larger.

Assume first $\ell_r=q$.
For $0<j\le p$ we have
\begin{equation*}
[e_{s-p},e_{2q+j}]
=[e_{s-p},[e_{2q-p+j},e_p]]
=[e_{s-p},e_{2q-p+j},e_p]-[e_{s-p},e_p,e_{2q-p+j}].
\end{equation*}
The first Lie product of the difference at the right-hand side is a scalar multiple of $[e_{s+2q-2p+j},e_p]$, and hence vanishes for $0<j<p$ because $\ell_{r+1}\ge q$.
The second Lie product of the difference vanishes because $\beta_{s-p}=0$.
Consequently, for $0<j<p$ we have
\begin{align*}
0=[e_{s-p},e_{2q+j}]
&=\sum_{i=0}^{2q-p+j}(-1)^i \binom{2q-p+j}{i}
[e_{s-p+i},e_p,z^{2q-p+j-i}]
\\&=\sum_{i=0}^{2q-p+j}(-1)^i \binom{2q-p+j}{i}
\beta_{s-p+i}
e_{s+2q-p+j}.
\end{align*}
Because $\beta_{s-p+i}=0$ for $i=0$, or $p<i\le q$,
or $q+p<i\le 2q$ (due to $\ell_{r+1}\ge q$), and separating the two remaining portions of the summation range, we find
\begin{align*}
0&=\sum_{i=1}^p (-1)^i \binom{2q-p+j}{i} \lambda- \sum_{i=q+1}^{q+p} (-1)^i \binom{2q-p+j}{i} \lambda_{i-q}
\\&=
\lambda \sum_{i=1}^j (-1)^i \binom{j}{i}+\lambda- \sum_{i=1}^j (-1)^i \binom{j}{i} \lambda_i-\lambda_p
\\&=
-\sum_{i=1}^j (-1)^i \binom{j}{i} \lambda_i-\lambda_p,
\end{align*}
for $0<j<p$,
where we have applied Lucas' theorem in the second line, and the identity
$\sum_{i=1}^j (-1)^i \binom{j}{i}=-1$ in the third line.
Now an easy induction on $j$ yields
$\lambda_1=\cdots=\lambda_{p-1}=\lambda_p$,
which amounts to the desired conclusion that the $r$th constituent is ordinary.

Assuming now $\ell_r>q$, we need a different argument.
We first show that this implies $\ell_r>q+p$.
We preliminarily note that for $0\le j\le p$ we have
\begin{align*}
[e_{s-p+1},e_{q+j}]
&=
\sum_{i=0}^{q-p+j} (-1)^i \binom{q-p+j}{i}[e_{s-p+1+i},e_p,z^{q-p+j-i}]
\\&=
\lambda\sum_{i=0}^{j} (-1)^i \binom{q-p+j}{i}e_{s+q-p+1+j}
=0,
\end{align*}
where we have used that $\beta_{s-p+1+i}=0$ for $p\le i\le q$ (including $q$ due to $\ell_r>q$), then Lucas' theorem, and finally the identity
$\sum_{i=0}^j (-1)^i \binom{j}{i}=0$.
Now let the positive integer $j$ be smallest such that $[e_{s+q-p+1+j},e_p]\neq 0$, and suppose $j\le p$ for a contradiction.
We have
\begin{align*}
0=[e_{s-p+1},e_{q+j},e_p]
&=
[e_{s-p+1},e_p, e_{q+j}]-[e_{s-p+1},[e_{q+j},e_p]]
\\&=
\lambda[e_{s+1},e_{q+j}]
\\&=
(-1)^j\lambda[e_{s+q-p+1+j},e_p],
\end{align*}
where in the second line we have used $[e_{q+j},e_p]$ because $q+j<2q-p+1$,
and in the third line we have used the generalized Jacobi identity and the minimality of $j$.
This contradiction shows $\ell_r>q+p$, as desired.

Now we proceed to show that the $r$th constituent is ordinary.
The next calculation is similar to that for the case $\ell=q$, but with $e_{s-p}$ replaced by $e_{s+\ell_r-q-j}$.
Thus, for $0<j\le p$ we have
\begin{align*}
[e_{s+\ell_r-q-j},e_{2q+j}]
&=
[e_{s+\ell_r-q-j},[e_{2q-p+j},e_p]]
\\&=
[e_{s+\ell_r-q-j},e_{2q-p+j},e_p]-[e_{s+\ell_r-q-j},e_p,e_{2q-p+j}].
\end{align*}
The first Lie product of the difference in the second line is a scalar multiple of $[e_{s+\ell_r+q-p},e_p]$, and hence vanishes because $\ell_{r+1}\ge q$.
The second Lie product in the difference vanishes because $\ell_r>q+p$.
Consequently, for $0<j\le p$ we have

\begin{align*}
0=[e_{s+\ell_r-q-j},e_{2q+j}]
&=\sum_{i=0}^{2q-p+j}(-1)^i \binom{2q-p+j}{i}
[e_{s+\ell_r-q-j+i},e_p,z^{2q-p+j-i}]
\\&=\sum_{i=0}^{2q-p+j}(-1)^i \binom{2q-p+j}{i}
\beta_{s+\ell_r-q-j+i}
e_{s+\ell_r+q}.
\end{align*}
We have $\beta_{s+\ell_r-q-j+i}=0$ in the range considered except possibly when
$q-p+j<i\le q+j$, where $\beta_{s+\ell_r-q-j+i}=\lambda_{-q+p-j+i}$.
Consequently, and restricting the range to $0<j<p$ for the last step, we find
\begin{align*}
0&=
\sum_{i=q+j-p+1}^{q+j} (-1)^i \binom{2q-p+j}{i} \lambda_{-q+p-j+i}
\\&=
\sum_{i=0}^{p-1} (-1)^{q+j-i} \binom{2q-p+j}{q+j-i} \lambda_{p-i}
\\&=
(-1)^{j+1}\sum_{i=0}^j (-1)^i \binom{j}{i} \lambda_{p-i}
\end{align*}
where we replaced $i$ with $q+j-i$ in the second line, and applied Lucas' theorem and the binomial identity $\binom{j}{j-i}=\binom{j}{i}$ in the third line.
This conclusion, for $0<j<p$, provides $p-1$ linearly independent homogeneous linear equations for $\lambda_1,\ldots,\lambda_p$,
and because $\sum_{i=0}^j (-1)^i \binom{j}{i}=0$ any equal values for $\lambda_1,\ldots,\lambda_p$ form a solution.
We conclude that $\lambda_1,\ldots,\lambda_p$ are all equal, and hence the $r$th constituent is ordinary, as desired.

\section{Construction of the exceptional algebras}\label{sec:exceptional}

In this section we construct the algebras of the exceptional family $\mathcal{E}$ in Theorem~\ref{thm:classification},
whose first constituent length $\ell$ satisfies $q<\ell<q+p$.
There will be one algebra for each value of a parameter $m$ in the range $0<m<p$, but when $m=p-1$ that algebra will actually not be exceptional,
being a translate of a subalgebra of an uncovered algebra of type $1$,
see Remark~\ref{rem:m=p-1}.
Our conclusions are summarized in Theorem~\ref{thm:exceptional}.
This section constitutes a proof of that result,
but supplements it with the additional information that all constituents past the first of those algebras are ordinary and end with the same constant.

Let $K=F(t)$ be the field of rational functions on an indeterminate $t$ over $F$, and consider the ring of divided powers $K[x;c]$, where $q=p^c$.
Thus, a basis of $K[x;c]$ consists of the monomials $x^{(i)}$, with $0\le i<q$ (writing $1$ for $x^{(0)}$), with multiplication given by
$x^{(i)}x^{(j)}=\binom{i+j}{i}x^{(i+j)}$
and extended linearly.
Note that the binomial coefficient vanishes modulo $p$ when $i,j<q$ but $i+j\ge q$, and so the right-hand side may be read as zero in that case.
Also, it will be both safe and convenient for our calculations to allow divided power $x^{(j)}$ with negative $j$ and interpret them as zero.

Let $\partial$ be the standard derivation of $K[x;c]$, hence $\partial x^{(i)}=x^{(i-1)}$ for $0<i<q$ and $\partial 1=0$,
and let $I$ denote the identity map of $K[x;e]$.
View $K[x;c]$ as an abelian Lie algebra, and $\partial$ as an element of the general Lie algebra $\mathrm{gl}(K[x;c])$.
We first look at the Lie subalgebra of $\mathrm{gl}(K[x;c])$ generated by
$Z=-\partial-tx^{(q-1)}I$
and $tx^{(q-p)}I$.
Thus,
$Zx^{(i)}=-x^{(i-1)}$ for $0<i<q$, but $Z1=-tx^{(q-1)}$.
The reason for the negative signs here is to avoid some awkward alternating signs later, due to our use of left-normed long Lie products
conflicting with the prevalent notation where derivations act on the left of their arguments.

We have
$[tx^{(q-p)}I,Z]
=-tx^{(q-p)}I\partial+\partial(tx^{(q-p)}I)
=tx^{(q-p-1)}I$,
and then, by an easy induction,
\[
[tx^{(q-p)}I,Z^j]
=tx^{(q-p-j)}I
\]
for $0<j\le q-p$.
(Here $Z^j$ refers to our long Lie product convention, not to a compositional power of $Z$.)
In particular, for $j=q-p$ we get $tI$, and hence
$tx^{(q-p)}I$ and $Z$ generate a Lie algebra of dimension $q-p+2$, nilpotent of class one less.
Also, if we give the vector space (or abelian Lie algebra) $K[x;c]$ a cyclic grading over the integers modulo $q$ by assigning degree $j$ (modulo $q$) to $x^{(q-j)}$ ,
then $Z$ and $tx^{(q-p)}I$ are graded derivations of $K[x;c]$, of degrees $1$ and $p$, respectively.

Now fix $0<m<p$,
and consider the semidirect sum of $K[x;c]$ and $\mathrm{gl}(K[x;c])$, hence with product
$[(f,A),(f',A')]=(Af'-A'f,[A,A'])$.
Inside that, consider the $F$-Lie subalgebra $L$ generated by
\[
z=(0,Z)
\qquad\text{and}\qquad
e_p=(x^{(q+m-p)},tx^{(q-p)}I).
\]
Then setting $e_j:=[e_p,z^{j-p}]$ for $j>p$, we have
\[
e_j=(x^{(q+m-j)},tx^{(q-j)}I)
\]
for $p\le j\le q+m$,
where the second entry is to be read as zero for $q<j\le q+m$ according to our stipulation on divided powers with negative exponents.
In particular, we have
$e_{q+m}=(1,0)$,
whence
$e_{q+m+1}=(tx^{(q-1)},0)$.
More generally, we have
\[
e_{rq+m+j}=(t^rx^{(q-j)},0)
\]
for $r>0$ and $0<j\le q$.

When we view the algebra $K[x;c]$
of divided powers
as an abelian Lie algebra over $F$ we can refine a shift of its $\Z/q\Z$-grading over $K$ described earlier to a $\Z$-grading,
by assigning degree $rq+m+j$ to $(t^rx^{(q-j)},0)$.
Then $z$ and $e_p$ act on $K[x;c]$ as graded derivations of degree $1$ and $p$ in this $\Z$-grading.
In particular, before doing any explicit calculation we see that
$[e_p,e_j]$ is an $F$-scalar multiple of $e_{p+j}$ for all $j\ge q$.
Together with previous information on the $\Z/q\Z$-grading we see that this extends to $j\ge p$.
Consequently, $L$ is a $\Z$-graded Lie algebra of maximal class, generated by elements of degree $1$ and $p$,
hence an algebra of type $p$ according to our terminology.

Now we look at the adjoint action of $e_p$ to find the constituent lengths.
For $p<j\le q-p+m$ we have
\begin{align*}
[e_j,e_p]
&=
[(x^{(q+m-j)},tx^{(q-j)}I),(x^{(q+m-p)},tx^{(q-p)}I)]
\\&=
(tx^{(q-j)}x^{(q+m-p)}-tx^{(q-p)}x^{(q+m-j)},0)
=0.
\end{align*}
Next, for $0\le j<p$ we have
\begin{align*}
[e_{q+m-j},e_p]
&=
[(x^{(j)},tx^{(j-m)}I),(x^{(q+m-p)},tx^{(q-p)}I)]
\\&=
\left(\binom{q-p+j}{j-m}-\binom{q-p+j}{j}\right)
(tx^{(q-p+j)},0)
\\&=
\left(\binom{j}{m}-1\right)
e_{q+m+p-j}.
\end{align*}
Because $\binom{p-1}{m}\equiv(-1)^m\pmod{p}$ and
$\binom{p-2}{m}=\frac{p-m-1}{p-1}\binom{p-1}{m}\equiv(-1)^m(m+1)\pmod{p}$,
the highest $j$ in the considered range $0\le j<p$ for which $\binom{j}{m}-1$ does not vanish modulo $p$ is $j=p-1$ when $m$ is odd,
and $j=p-2$ when $m$ is even.
Consequently, the first constituent of $L$ has length $q+m$ when $m$ is odd, and $q+m+1$ when $m$ is even.
We discuss the significance of this discrepancy in Remark~\ref{rem:fake_constituents},
and give an alternate description of the entries of the first constituent in Remark~\ref{rem:alternate}.

Now we look at constituents past the first.
For $r>0$ we have
\begin{align*}
[e_{rq+m+j},e_p]
&=
[(t^rx^{(q-j)},0),(x^{(q+m-p)},tx^{(q-p)}I)]
\\&=
-\binom{2q-p-j}{q-p}
(t^{r+1}x^{(2q-p-j)},0),
\end{align*}
whence
$[e_{rq+m+j},e_p]=0$ for $0<j\le q-p$,
and
$[e_{rq+m+j},e_p]=-e_{rq+m+p+j}$
for $q-p<j\le q$,
because
\[
\binom{2q-p-j}{q-p}
\equiv
\binom{q-p}{q-p}\binom{q-j}{0}=1
\pmod{p}.
\]
Consequently, when $m$ is odd all constituents past the first are ordinary of length $q$, ending in $-1$ (with our non-normalized choice of generators).
However, when $m$ is even all constituents past the first are ordinary, ending in $-1$,
but only those from the third on have length $q$,
with the second constituent having length $q-1$.

\begin{rem}\label{rem:fake_constituents}
The constituent structure of the algebras of the family $\mathcal{E}$ constructed here
would be marginally simpler to describe in terms of Scarbolo's definition of constituents in~\cite{Sca:thesis},
which we discussed at the end of Section~\ref{sec:const_lengths}:
irrespectively of the parity of $m$, Scarbolo's first constituent would have length $q+m$,
and all constituents past the first would be ordinary of length $q$.
\end{rem}

\begin{rem}\label{rem:alternate}
We give an alternate description of the first constituent of the algebra $L$ by computing the polynomial
$g(x)=\beta_{\ell-p+1}x^{p-1}+\beta_{\ell-p+2}x^{p-2}+\cdots+\beta_{\ell}$,
which was crucial in our proof of Theorem~\ref{thm:first_length}.
Differently from there, $g(x)$ will not be monic, as we have not applied the normalization $\beta_{\ell-p+1}=1$.
Note that $\binom{j}{m}\equiv(-1)^{m+j}\binom{p-1-m}{p-1-j}\pmod{p}$ for $m$ and $j$ in the given range.
Assuming first $m$ odd, where $\ell=q+m$, we have
\begin{align*}
g(x)
&=
\sum_{j=0}^{p-1}(-x)^j+\sum_{j=0}^{p-1}(-1)^{j-m}\binom{p-1-m}{j-m}x^j
\\&=
-(1-x)^{p-1}+
x^m(1-x)^{p-1-m},
\end{align*}
that is,
$g(x)=(1-x)^{p-1-m}\bigl(x^m-(1-x)^m\bigr)$.
When $m$ is even the above polynomial has degree $p-2$, hence $\ell=q+m+1$ as noted above,
and the actual $g(x)$ of the proof of Theorem~\ref{thm:first_length} is obtained
by multiplying it by $x$, thus finding
$g(x)=x(1-x)^{p-1-m}\bigl(x^m-(1-x)^m\bigr)$
in this case.
Note that in both cases $g(x)$ is a multiple of $(x-1)^{p-k_0}$
(where $k_0$, as defined in Subsection~\ref{subsec:spurious}, equals $m+1$ if $m$ is odd and $m+2$ if $m$ is even),
matching what we proved in Subsection~\ref{subsec:further_info}.
\end{rem}

\begin{rem}\label{rem:m=p-1}
When $m=p-1$ the algebra $L$ described in this section is not exceptional, being a translate of a subalgebra of an algebra of type $1$.
In fact, the sequence $(\beta'_i)_{i>p}$ of the translated algebra $L(1)$
satisfies
\[
\beta'_{iq+j}=
\begin{cases}
1&\text{$0<j\le q-p$},
\\
0&\text{$q-p<j\le q$}.
\end{cases}
\]
Because each block of entries equal to $1$ has length a multiple of $p$, all constituents of $L(1)$ are ordinary.
Consequently, $L(1)$ isomorphic to a subalgebra of an uncovered algebra $N$ of type $1$, according to Proposition~\ref{OrdinaryEnding}.
The sequence of constituent lengths of $L(1)$, and hence of $N$ as well, is
\[
2p,p^{q/p-2},2p-1,(p^{q/p-2},2p)^{\infty},
\]
where exponents, as in~\cite{CMN,CN}, denote repetition of that constituent length, or block of lengths, the specified number of times.
In the notation of~\cite{CMN,CN} the algebra $N$ is isomorphic to $AFS(1,h,\infty,p)$.
When $0<m<p-1$ one easily sees that no translate $L(\delta)$ of $L$ can have only ordinary constituents,
with the obstacle arising from the final entries of the first constituent of $L$.
\end{rem}

\bibliography{References}

\end{document}